\documentclass{amsart}
\usepackage[utf8]{inputenc}  
\usepackage[T1]{fontenc}
\usepackage[english]{babel}     
\usepackage[a4paper, left=1in, right=1in, top=1in, bottom=1in]{geometry}
\usepackage{indentfirst}
\usepackage{tikz-cd}
\usepackage{graphicx}           
\usepackage{amsmath,amssymb}    
\usepackage{amsthm}             
\usepackage{mathtools}          
\usepackage{enumitem}           
\usepackage{listings}           
\usepackage{todonotes}          
\usepackage{hyperref}           
\usepackage[cal=boondox]{mathalfa}
\usepackage{relsize}
\usepackage{stmaryrd}
\usepackage{setspace}
\usepackage{chngcntr}

\counterwithin{equation}{subsection}
\counterwithin{figure}{subsection}

\setenumerate[1]{label={(\arabic*)}} 
\setlist[enumerate]{itemsep=0pt}
\newtheorem{innercustomthm}{Theorem}
\newenvironment{mythm}[1]
  {\renewcommand\theinnercustomthm{#1}\innercustomthm}
  {\endinnercustomthm}

\setlist[itemize]{itemsep=0pt}

\newcommand\restr[2]{{
  \left.\kern-\nulldelimiterspace 
  #1 
  \right |_{#2} 
  }}

\newtheorem{theorem}{Theorem}[section]
\newtheorem{corollary}{Corollary}[theorem]
\newtheorem{lemma}[theorem]{Lemma}
\newtheorem{proposition}[theorem]{Proposition}
\newtheorem*{claim}{Claim}

\newtheorem{thmx}{Theorem}

\theoremstyle{definition}
\newtheorem{definition}[theorem]{Definition}

\newtheorem*{remark}{Remark(s)}

\begin{document}

\title{Versal deformation of transversely holomorphic flows on the boundary of strongly convex domains of $\mathbb C^n$}
\author{Mounib Abouanass}
\date{\today}

\maketitle
\begin{abstract}
    In this article, we give a versal deformation for any transversely holomorphic foliation $\mathcal{F}_0$ given by the intersection of the orbits of a holomorphic vector field $\xi$ defined on a neighborhood of the closure of a bounded strongly convex open domain $\Omega \subset\mathbb C^n$ ($n\geq2$) with smooth boundary, with its boundary $\partial \Omega$.

    That is, any germ of deformation of $\mathcal{F}_0$ is also obtained by intersecting the orbits of a deformation of $\xi$ with the boundary of $\Omega$.
\end{abstract}
\tableofcontents
\section{Introduction}
Smooth flows have long been of interest to both mathematicians and physicists. Their study can be approached either dynamically (by considering the one-parameter subgroup of smooth diffeomorphisms and using various tools from dynamical systems and ergodic theory) or geometrically (by examining the partition of the phase space into orbits, i.e., the orbit foliation). A classical approach involves assuming transverse structures for the flow.

For example, Brunella and Ghys (see \cite{brunella_umbilical_1995}, \cite{brunella_transversely_1996}, \cite{ghys_transversely_1996}) have studied transversely holomorphic flows on smooth three-manifolds, i.e., flows whose holonomy pseudo-group consists of biholomorphic maps between open subsets of $\mathbb C$. They achieved a complete classification using advanced topological and analytical techniques. Among the examples they studied are Poincaré foliations on $S^3$, which are those induced by the singularity of a holomorphic vector field in $\mathbb C^2$ within the Poincaré domain, along with their finite quotients. In fact, such examples exhaust all transversely holomorphic flows on $S^3$.

The aim of this paper is to study, in a similar spirit, transversely holomorphic foliations arising from the intersection of the orbits of a holomorphic vector field, defined in a neighborhood of the closure of a bounded strongly convex domain $\Omega \subset\mathbb C^n$ with smooth boundary (where $n \geq 2$), with its boundary $\partial M$.
More precisely, we study deformations of such transversely holomorphic foliations as defined by Haefliger, Girbau, and Sundararaman in \cite{haefliger_deformations_1983}, as a continuation of the work of Kodaira and Spencer on the deformation of complex and, more generally, pseudo-group structures (see, for example, \cite{kodaira_deformations_1958} and \cite{kodaira_multifoliate_1961}).

We rely on the work of Brunella (see \cite{brunella_remarque_1994}) which allows us to simplify the situation and consider simply the example of the closed unit ball $\overline{B^n}$, and the results of Ito (see \cite{ito_poincare-bendixson_1994}, \cite{ito_number_1996}) regarding  the intersection of such a holomorphic vector field with the unit sphere $S^{2n-1}=\partial\overline{B^n}$. Furthermore, we follow Haefliger’s proof (see \cite{haefliger_deformations_1985}) on the deformation of a transversely holomorphic flow obtained by intersecting the sphere $S^{2n-1}$ with the orbits of a holomorphic flow that has $0$ as a contracting fixed point.

The main result is :

\begin{thmx}
\label{thm:defo}
    Let $\mathcal{F}_0$ a transversely holomorphic foliation on the boundary $\partial \Omega$ of a bounded strongly convex domain $\Omega \subset\mathbb C^n$ with smooth boundary, obtained by intersecting with $\partial \Omega$ the orbit foliation $\mathcal{F}$ of a holomorphic vector field $\xi$ defined on a neighborhood of $\overline{\Omega}$.\\
    Then there exists a holomorphic diffeomorphism $\Phi$ from a neighborhood of $\overline{\Omega}$ to a neighborhood of the unit closed ball $\overline{B^n}$ such that if we note $\lambda=(\lambda_1, \ldots, \lambda_n)$ the eigenvalues of the differential at $0$ of $\Phi_*\xi$, and if $S$ is a sufficiently small open neighborhood of $0$ in the space $\mathcal{g}_\lambda$ - of holomorphic vector fields on $\mathbb C^n$ commuting with $\sum_{j=1}^n\lambda_jz_j \frac{\partial}{\partial z_j}$ - complementary to the subspace in $\mathcal{g}_\lambda$ generated by $\Phi_*\xi$ and $L_{\Phi_*\xi}(\mathcal{g}_\lambda)$, then the family $\mathcal{F}_0^S=\left(\mathcal{F}_0(\xi+\Phi^*s) \right)_{s \in S}$ of transversely holomorphic foliations on $\partial \Omega$ coming from the intersections with the boundary $\partial M$ of the holomorphic foliations $(\mathcal{F}(\xi +\Phi^*s))_{s\in S}$ induced by the family of holomorphic vector fields $(\xi +\Phi^*s)_{s\in S}$ is a versal deformation of the transversely holomorphic foliation $\mathcal{F}_0$ parametrized by $(S,0)$. \\
    The latter means that for any germ $\mathcal{F}_0^{S'}$ of deformation of $\mathcal{F}_0$ parametrized by the germ of an analytic space $(S', 0)$, there exists an analytic map $\varphi: (S', 0) \to(S, 0)$ so that $\mathcal{F}_0^{S'}$ is isomorphic to $(\mathcal{F}_0^{\varphi(s')})_{s \in S}$.
\end{thmx}

This result is a generalization of the one in \cite{haefliger_deformations_1985} since we consider bounded strongly convex domains with smooth boundary instead of the unit ball of $\mathbb C^n$, and since we don't have to assume anything on the singularities of the holomorphic vector field $\xi$ inducing $\mathcal{F}_0$ (i.e. $\xi$ need not have $0$ as a contracting fixed point).\\

In the first section, we recall the general theory of the deformation of transversely holomorphic foliations (\cite{haefliger_deformations_1983}, \cite{duchamp_deformation_1979}), summarize key results, and consider the special case of a holomorphic family of vector fields.\\
With the help of Brunella and Ito’s results, we simplify the problem in Section \ref{sec:3}, which in turn allows us, in Section \ref{sec:4}, to restate the main theorem \ref{thm:defo}, after recalling the notions of resonance for holomorphic vector fields as discussed by Arn’old in \cite{szucs_geometrical_1996}.\\
Eventually, in Section \ref{sec:5}, we prove the main theorem by following the proof of \cite{haefliger_deformations_1985} in case of a contracting holomorphic vector field.

\section{Deformation theory of (transversely) holomorphic foliations}
\subsection{Transversely holomorphic foliations}
\subsubsection{Definitions}
We first recall some definitions and fundamental results by Haefliger, Girbau and Sundararaman (see \cite{haefliger_deformations_1983}) and Kalka and Duchamp (see \cite{duchamp_deformation_1979}, \cite{duchamp_holomorphic_1984}). See also \cite{gomez-mont_transversal_1980}.\\
A (smooth) \textit{transversely holomorphic foliation} of complex codimension $p$ on a smooth manifold $M$ of dimension $n$ can be given equivalently by one of the following equivalent data :

\begin{enumerate}
    \item a smooth atlas $(U_i, \psi_i)_i$ on $M$ satisfying :
            \begin{enumerate}
                \item For all $i$, $\psi_i(U_i)=U_i^1 \times U_i^2$, where $U_i^1$ and $U_i^2$ are connected open subsets of $\mathbb R^{n-2p}$ and $\mathbb C^p$ respectively ;
                \item For all $i,j$, there exist maps $f_{ij}$ and $h_{ij}$ such that
                \[\forall(x,z) \in \psi_i(U_i \cap U_j) \subset \mathbb R^{n-2p} \times \mathbb C^p, \; \psi_i \circ \psi_j^{-1}(x,z)=(f_{ij}(x,z), h_{ij}(z))\]
                with $h_{ij}$ holomorphic.
            \end{enumerate}
    \item an open covering $(U_i)_i$ of $M$ and a family of smooth submersions $s_i:U_i \to \mathbb C^p$ such that for $i,j$, there exists a biholomorphic map $g_{ij}:s_j(U_i\cap U_j) \to s_i(U_i \cap U_j)$ satisfying :
    $$s_i=g_{ij}\circ s_j \quad \text{ on } U_i\cap U_j.$$
\end{enumerate}
The second definition best captures the transversal nature of the foliation, particularly its transversal holomorphic aspect.
By abuse of language, we will call a distinguished chart an element $(U_i, \psi_i)$  or $(U_i, f_i)$ depending on the context.\\
Let $(S,0)$ a germ of an analytic space, that is the germ at $0$ of an analytic set $S \subset \mathbb C^k$ containing $0$. 
Let $J_{S,0}$ the ideal of germs $f \in \mathcal{O}_{\mathbb C^k,0}$ which vanish on $(S,0)$.
The ring $\mathcal{O}_{S,0}$ of germs of functions on $(S,0)$ which can be extended as germs of holomorphic functions on $(\mathbb C^k,0)$ is by definition the quotient of $\mathcal{O}_{\mathbb C^k,0}$ by $J_{S,0}$ (see \cite{demailly_complex_2012}).\\
A \textit{smooth function} $g$ on $(S,0)$ is represented by a smooth function $G$ on a neighborhood of $0$ in $\mathbb C^k$ and two such functions $G, G'$ represent the same $g$ if the germ of $G-G'$ is in the ideal generated by $J_{S,0}$ in the ring of germs of smooth functions on $\mathbb C^k$ at $0$.\\
Let $M$ a smooth manifold. A \textit{smooth function} $g^S$ on $(S,0)\times M$ is a family of smooth functions $(g_x)_{x \in M}$ on $(S,0)$ varying smoothly with $x\in M$, i.e. for every $s\in S$, the map $g^s=x \mapsto g_x(s)$ is smooth on $M$.
Such a smooth function $g^S$ will be considered as a \textit{deformation} of $g^0:=(g_x(0))_{x \in M}$.
If $g^S$ takes value in $\mathbb C$, then it is said to be \textit{holomorphic in $s$} if for every $x \in M$, the function $g_x$ is holomorphic on $(S,0)$.
We define in the same way smooth maps on $(S,0)\times M$ to $\mathbb R^m$ as well as smooth maps on $(S,0)\times M$ to $\mathbb C^m$  holomorphic in $s$.\\
Now, consider a transversely holomorphic foliation $\mathcal{F}$ on $M$ of complex codimension $p$. \textit{A germ of deformation $\mathcal{F}^S$ of $\mathcal{F}$ parametrized by $(S,0)$} is the data of an open covering of $(U_i)_i$ and for each $i$ a smooth map $f_i^S$ on $(S,0)\times U_i$ holomorphic in $s$ which is, for $s\in (S,0)$ fixed, a submersion such that for $i,j$, there exists a holomorphic family $(g^s_{ij})_{s\in (S,0)}$ of biholomorphic maps $g_{ij}^s:f_j^s(U_i\cap U_j) \to f_i^s(U_i\cap U_j)$ satisfying :
$$f_i^s=g_{ij}^s\circ f_j^s \quad \text{ on } U_i \cap U_j.$$
Moreover, the transversely holomorphic foliation $\mathcal{F}$ can be defined by the open covering $(U_i)_i$, the family of submersions $(f_i^0)_i$ and the family of biholomorphic transition maps $(g_{ij}^0)_{ij}$.\\
If $\mathcal{F}^S$ and $\mathcal{F'}^S$ are two deformations of $\mathcal{F}$ parametrized by the same germ of space $(S,0)$, they are \textit{isomorphic} if there exists a smooth map $h^S=(h^s)_{s \in (S,0)}$ on $(S,0)\times M$ consisting of diffeomorphisms $h^s:M \to M$ such that for each $s\in(S,0)$, $\mathcal{F}^s$ and $(h^s)^*\mathcal{F'}^s$ define the same transversely holomorphic foliation.
That is, if the deformation $\mathcal{F'}^S$ and $\mathcal{F}^S$ are given by a common covering $(W_i)_i$ and, for each $i$, respectively by a smooth map $f_i'^s$ on $(S,0) \times W_i$ holomorphic in $s$ which is a submersion for each $s$ and a smooth map $f_i^s$ on $(S,0) \times W_i$ holomorphic in $s$ which is a submersion for each $s$, then there exists a holomorphic family $(g^s_{ij})_{s\in (S,0)}$ of biholomorphic maps $g_{ij}^s$ on open sets of $\mathbb C^p$ satisfying :
$$f_i^s=g_{ij}^s\circ (f_j^{'s}\circ h^s)$$
when defined.\\
If $\varphi : (S',0) \to (S,0)$ is a analytic morphism and $\mathcal{F}^S$ is a deformation of $\mathcal{F}$ parametrized by $(S,0)$, then $\mathcal{F}^{\varphi(S')}=(\mathcal{F}^{\varphi(s')})_{s' \in (S',0)}$ is a deformation of $\mathcal{F}$ parametrized by $(S',0)$ and is called \textit{the deformation induced by $\varphi$.}
\begin{remark}
    \leavevmode
    \begin{itemize}
        \item If $\mathcal{F'}^S=(\mathcal{F'}^s)_{s\in S}$ is a germ of deformation of $\mathcal{F}'_0$ parametrized by $(S,0)$ on $M'$, and $h=(h_s)_{s \in S}$ a smooth family of diffeomorphisms from $M$ to $M'$, then $h^*\mathcal{F}'^S:=(h_s^*(\mathcal{F'}^s))_{s\in S}$ is a germ of deformation of $h_0^*(\mathcal{F}_0')$ parametrized by $(S,0)$ on $M$.\\
        We can also define a germ of deformation of $(h_0)_*\mathcal{F}_0$ parametrized by $(S,0)$ on $M'$ by $h_*\mathcal{F}^S:=((h_s)_*\mathcal{F}^s)_{s \in S}$ if $\mathcal{F}^S$ is a germ of deformation of $\mathcal{F}_0$ parametrized by $(S,0)$ on $M$.
        \item With the previous notations, let $\mathcal{F}^S$ and $\mathcal{F}'^S$ be two germs of deformation of $\mathcal{F}_0$ on $M$ parametrized by $(S,0)$. \\
        Then $\mathcal{F}^S$ and $\mathcal{F}'^S$ are isomorphic if and only if $h_*\mathcal{F}^S$ and $h_*\mathcal{F}'^S$ are isomorphic.
    \end{itemize}
     
\end{remark}
\subsubsection{Some sheaves}
Consider a transversely holomorphic foliation $\mathcal{F}$ of complex codimension $p$ on a smooth manifold $M$.
\begin{definition}
    The sheaf $\sigma^{tr}_{\mathcal{F}}$ of \textit{transversely holomorphic functions of $\mathcal{F}$} on $M$ is defined as the unique sheaf on $M$ whose restriction at a distinguished chart $(U_i,f_i)$ for the transversely holomorphic foliation is precisely the pullback of the sheaf of holomorphic functions on $\mathbb C^p$ by $f_i$, that is 
$$\restr{\sigma^{tr}_{\mathcal{F}}}{U_i}=f_i^{-1}\sigma_{\mathbb C^p}$$
where $\sigma_{\mathbb C^p}$ is the sheaf of holomorphic functions on $\mathbb C^p$ (see \cite{perrin_algebraic_2008} for the existence and uniqueness of such a sheaf).\\
This sheaf is called the structural sheaf of the transversely holomorphic foliation $\mathcal{F}$ (see \cite{gomez-mont_transversal_1980} for another equivalent definition).
\end{definition}
\begin{definition}
    The sheaf $\eta_{\mathcal{F}}$ of \textit{infinitesimal automorphisms of $\mathcal{F}$} on $M$ is defined as the sheaf on $M$ whose sections are smooth vector fields giving local flows which are isomorphisms of $\mathcal{F}$, that is which pull-back $\mathcal{F}$ on $\mathcal{F}$. 
    It can also be seen as the unique sheaf on $M$ whose restriction at a distinguished chart $(U_i,\psi_i)$ for the transversely holomorphic foliation is precisely the pullback by $\psi_i$ of the sheaf of smooth vector fields $X$ which writes as
    $$X(x,z)=\sum_{i=1}^k a_i(x,z)\dfrac{\partial}{\partial x_i} + \sum_{j=1}^p b_j(z)\dfrac{\partial}{\partial z_j} +\overline{b}_j(z)\dfrac{\partial}{\partial \overline{z_j}} $$
    where each $b_j$ is holomorphic.
    \end{definition}
\begin{definition}
    The sheaf $\theta^{tr}_{\mathcal{F}}$ of\textit{ transversely holomorphic vector fields of $\mathcal{F}$} on $M$ is defined as the quotient of the sheaf $\theta_{\mathcal{F}}$ by the sheaf of smooth vector fields tangent to the leaves of $\mathcal{F}$.
    It can also be seen as the unique sheaf on $M$ whose restriction at a distinguished chart $(U_i,f_i)$ for the transversely holomorphic foliation is precisely the pullback by $f_i$ of the sheaf of holomorphic vector fields on $\mathbb C^p$, that is 
$$\restr{\theta^{tr}_{\mathcal{F}}}{U_i}=f_i^{-1}\theta_{\mathbb C^p}$$
where $\theta_{\mathbb C^p}$ is the sheaf of holomorphic vector fields on $\mathbb C^p$ (again, see \cite{gomez-mont_transversal_1980} for another equivalent definition).
    \end{definition}
Haefliger, Jirbau and Sundararaman call it the \textit{fundamental sheaf}.
\begin{remark}
\leavevmode
    \begin{itemize}
        \item If we write as $\theta^{||}_\mathcal{F}$ the sheaf of smooth vector fields on $M$ which are tangent to the leaves of $\mathcal{F}$, then the following complex of sheaves is exact :
        \begin{align}
            0 \longrightarrow \theta^{||}_{\mathcal{F}} \longrightarrow \eta_\mathcal{F} \longrightarrow \theta^{tr}_{\mathcal{F}} \longrightarrow 0.
        \end{align}
        Therefore, as $\theta^{||}_\mathcal{F}$ is a fine sheaf, it comes for $k\in \mathbb N^*$, $$H^k(M,\eta_\mathcal{F})\cong H^k(M,\theta^{tr}_{\mathcal{F}}).$$
        \item Since each $f_i$ is a submersion thus an open map, we can also write $\restr{\theta^{tr}_{\mathcal{F}}}{U_i}$ as $$\left \{g\circ f_i, \quad g :f_i(U_i) \to \mathbb C^p \;\text{ is holomorphic}\right \}.$$
        \item For a distinguished chart $(U_i, f_i)$, 
        $$\restr{\theta^{tr}_{\mathcal{F}}}{U_i}=f_i^{-1}\theta_{\mathbb C^p}\cong f_i^{-1}{\sigma_{\mathbb C^p}}^{\oplus p}\cong(f_i^{-1}\sigma_{\mathbb C^p})^{\oplus p}=\restr{(\sigma^{tr}_{\mathcal{F}})^{\oplus p}}{U_i}.$$
    \end{itemize}
    
\end{remark}
\begin{definition}
    Consider a germ of deformation $\mathcal{F}^S$ of a transversely holomorphic foliation $\mathcal{F}$ on $M$ parametrized by $(S,0)$.\\
    There is a well-defined linear map $\rho$ from the tangent space of $S$ at $0$ to the first cohomology group of the sheaf of transversely holomorphic vector fields of $\mathcal{F}$ on $M$, usually called the \textit{Kodaira-Spencer map}, defined as follow :
    the vector $\restr{\dfrac{\partial}{\partial s}}{0} \in T_0S$ is mapped to the cohomology class of the cocycle associating to $U_i\cap U_j$ the section $\left ( \restr{\dfrac{\partial g_{ij}^s}{\partial s}}{0} \right )\circ f_j^0$ of $\restr{\theta^{tr}_{\mathcal{F}}}{U_i \cap U_j}$.
\end{definition}
\subsubsection{Versal deformation}
In \cite{haefliger_deformations_1983}, the authors proved the following result which is a version of the Kodaira-Spencer-Kuranishi theorem for transversely holomorphic foliations on compact manifolds:
\begin{theorem}
    \label{thm:2.5}
    Let $\mathcal{F}$ a transversely holomorphic foliation on a compact manifold $M$.\\
    Then there is a (unique up to isomorphism) germ $(S,0)$ of analytic space parameterizing a \textit{versal} germ of deformation $\mathcal{F}^S$ of $\mathcal{F}$. That is, for any germ of deformation $\mathcal{F'}^{S'}$ of $\mathcal{F}$ parametrized by an analytic space $(S',0)$, there exists an analytic map $\varphi : (S',0) \to (S,0)$ such that $\mathcal{F'}^{S'}$ is isomorphic to $\mathcal{F}^{\varphi(S')}$.\\
    Moreover, the differential $d_0\varphi$ of $\varphi$ at $0$ is unique.
\end{theorem}
\begin{corollary}
    \label{cor:2.5.1}
    With the previous assumptions, if $\mathcal{F}^{S'}$ is a germ of deformation of $\mathcal{F}$ parametrized by a non singular analytic space $(S',0)$ such that the Kodaira-Spencer map $\rho : T_0S' \to H^1(X,\theta^{tr}_{\mathcal{F}})$ is an isomorphism, then the given map $\varphi : (S',0) \to (S,0)$ is an isomorphism.
\end{corollary}
\begin{remark}
     By the previous remarks, if $\mathcal{F}^S$ is a versal germ of deformation of $\mathcal{F}_0$ parametrized by $(S,0)$ on $M$, and $h=(h_s)_{s \in S}$ a smooth family of diffeomorphisms from $M$ to $M'$, then $h_*(\mathcal{F}^S)$ is a versal germ of deformation of $(h_0)_*(\mathcal{F}_0)$ parametrized by $(S,0)$ on $M'$.
\end{remark}
On the other hand, Kalka and Duchamp have constructed an elliptic resolution for each sheaf $\sigma^{tr}_\mathcal{F}$ and $\theta^{tr}_\mathcal{F}$. Using results coming from the theory of elliptic complexes (see \cite{wells_differential_2008}), they prove that their cohomology groups are all finite dimensional. We will use these results and more than that. \\
The approach of Kalka and Duchamp is the following. Consider a transversely holomorphic foliation $\mathcal{F}$ of complex codimension $p$ on a smooth manifold $M$. Denote by $T\mathcal{F}$ tangent bundle to the leaves of $\mathcal{F}$ and by $Q:=TM/T\mathcal{F}$ the normal bundle of $\mathcal{F}$. We can define a complex structure on the bundle $Q$ by pulling back the natural complex structure on $\mathbb C^p$ via the submersions $f_i$. This complex structure induces a splitting of the complexified normal bundle $$Q^\mathbb C=Q^{(1,0)} \oplus Q^{(0,1)}$$
where $Q^{(0,1)}$ is the complex conjugate of $Q^{(1,0)}$. From the natural short exact sequence of vector bundles 
$$0 \longrightarrow T\mathcal{F} \longrightarrow TM \longrightarrow \nu \longrightarrow 0$$
comes another short exact sequence of (complex) vector bundles
$$0 \longrightarrow E \longrightarrow T_{\mathbb C}M \longrightarrow Q^{(1,0)} \longrightarrow 0$$
where $E\cong T_{\mathbb C}\mathcal{F} \oplus Q^{(0,1)}$.
$E$ is a complex subbundle of $T_{\mathbb C}M$ satisfying $$T_{\mathbb C}M=E+\overline{E} \quad \text{ and } \quad [\Gamma(E), \Gamma(E)] \subset \Gamma(E)$$
where $\Gamma(E)$ refers to the sheaf of sections of $E$. In fact the complex version of the Frobenius theorem gives that this data is equivalent to a transversely holomorphic foliation of complex codimension $p$ on $M$.
This allows them to find a one-to-one correspondence between transversely holomorphic foliations "near $\mathcal{F}$" and a particular subspace of $\text{Hom}_{\mathbb C}(E, Q^{(1,0)})$, denote by $\text{Fol}(\mathcal{F})$.
Eventually, they prove in \cite{duchamp_holomorphic_1984} that there exists a germ of analytic set $B\subset H^1(M, \theta^{tr}_\mathcal{F})$ at $0\in H^1(M, \theta^{tr}_\mathcal{F})$ and a holomorphic map $\Phi : B \to \text{Fol}(\mathcal{F})$ such that every holomorphic foliation "sufficiently near $\mathcal{F}$" is equivalent, via a diffeomorphism of $M$ near the identity, to an element in the image of $\Phi$.\\
If we note, for $s\geq1$, $E^{*s}$ the bundle $\bigwedge^s E^*$, and $d_E : E^{*s} \to E^{*s+1}$ given in local coordinates by differentiating a smooth complex valued form with respect to $(x, \overline{z})$, then:
\begin{itemize}
    \item the complex $(E^{*\bullet}, d_E)$ is elliptic ;
    \item the sequence
    $$0 \longrightarrow \sigma^{tr}_\mathcal{F}\longrightarrow \Gamma(E^{*0})\xrightarrow[]{d_E} \Gamma(E^{*1})\xrightarrow[]{d_E} \Gamma(E^{*2})\xrightarrow[]{d_E} \cdots $$
    is a resolution of the sheaf $\sigma^{tr}_\mathcal{F}$.
\end{itemize}
In the same way, if we note for $s \geq 1$ the bundle $E_Q^{*s}=E^{*s} \otimes_{\mathbb C}Q^{(1,0)}$ and $d_Q=d_E \otimes id : E_Q^{*s} \to E_Q^{*s+1}$, then:
\begin{itemize}
    \item the complex $(E_Q^{*\bullet}, d_Q)$ is elliptic ;
    \item the sequence
    $$0 \longrightarrow \theta^{tr}_\mathcal{F}\longrightarrow \Gamma(E_Q^{*0})\xrightarrow[]{d_Q} \Gamma(E_Q^{*1})\xrightarrow[]{d_Q} \Gamma(E_Q^{*2})\xrightarrow[]{d_Q} \cdots $$
    is a resolution of the sheaf $\theta^{tr}_\mathcal{F}$.
\end{itemize}
These facts will allow us to use results coming from the theory of elliptic complexes (see \cite{wells_differential_2008}) in our case.
\subsection{Holomorphic foliations}
All of what has been said for transversely holomorphic foliation can be adaptated and stated for holomorphic foliations on complex manifold.\\
A \textit{holomorphic foliation} of (complex) codimension $p$ on a complex manifold $M$ of (complex) dimension $n$ can be given equivalently by one of the following equivalent data :

\begin{enumerate}
    \item a holomorphic atlas $(U_i, \psi_i)_i$ on $M$ satisfying
        \begin{enumerate}
            \item For all $i$, $\psi_i(U_i)=U_i^1 \times U_i^2$, where $U_i^1$ and $U_i^2$ are connected open subsets of $\mathbb C^{n-p}$ and $\mathbb C^p$ respectively ;
            \item For all $i,j$, there exist (holomorphic) maps $f_{ij}$ and $h_{ij}$ such that
            \[\forall(w,z) \in \psi_i(U_i \cap U_j) \subset \mathbb C^{n-p} \times \mathbb C^p, \; \psi_i \circ \psi_j^{-1}(w,z)=(f_{ij}(w,z), h_{ij}(z)).\]
        \end{enumerate}
    \item an open covering $(U_i)_i$ of $M$ and a family of holomorphic submersions $f_i:U_i \to \mathbb C^p$ such that for $i,j$, there exists a biholomorphic map $h_{ij}:s_j(U_i\cap U_j) \to s_i(U_i \cap U_j)$ satisfying :
    $$s_i=h_{ij}\circ s_j \quad \text{ on } U_i\cap U_j.$$
\end{enumerate}
Remark that a holomorphic foliation on a complex manifold $M$ induces naturally a transversely holomorphic foliation on $M$ by considering its underlying smooth structure.\\
A deformation theory of holomorphic foliations on complex manifolds can similarly be formalized, as with transversely holomorphic foliations, by adjusting certain terms. 
Another definition of germ of deformation uses a family of holomorphic adapted atlases $(U_i,\psi_i^s)_i$ satisfying conditions analogous to the previous ones and holomorphic in $s \in S$.\\
The fundamental sheaf becomes the sheaf $\theta_\mathcal{F}$ of infinitesimal automorphisms of $\mathcal{F}$, that is the sheaf of holomorphic vector fields inducing local holomorphic flows which are isomorphisms of $\mathcal{F}$.\\
Also note that the sheaf of transversely holomorphic vector fields for the holomorphic foliation coincides with the sheaf of transversely holomorphic vector fields for the induced transversally holomorphic foliation.\\
The sheaf of infinitesimal automorphisms of $\mathcal{F}$ can be seen as the unique sheaf on $M$ whose restriction at a distinguished chart $(U_i,\psi_i)$ for the holomorphic foliation is precisely the pullback by $\psi_i$ of the sheaf of holomorphic vector fields $X$ which writes as
    $$X(w,z)=\sum_{i=1}^{n-p} a_i(w,z)\dfrac{\partial}{\partial w_i} + \sum_{j=1}^p b_j(z)\dfrac{\partial}{\partial z_j},$$
or equivalently whose restriction to $U_i$ is the set of pull-backs of such holomorphic vector fields by $\psi_i$.\\
In that case, if $\mathcal{F}^S$ is a germ of deformation of a holomorphic foliation $\mathcal{F}$ on $M$ parametrized by $(S,0)$, the Kodaira-Spencer map $\rho$ is defined as the map from the tangent space of $S$ at $0$ to the first cohomology group of the sheaf of infinitesimal automorphisms of $\mathcal{F}$ on $M$ which assigns to the vector $\restr{\dfrac{\partial}{\partial s}}{0} \in T_0S$ the cohomology class of the cocycle associating to $U_i \cap U_j$ the section $(d\psi_i^0)^{-1} \circ\left ( \restr{\dfrac{\partial g_{ij}^s}{\partial s}}{0} \right )\circ \psi_j^0$ of $\restr{\theta_{\mathcal{F}}}{U_i \cap U_j}$, where $g_{ij}^s=\psi_{i}^s\circ (\psi_{j}^s)^{-1}$\\
An analog result to Theorem \ref{thm:2.5} can be proven in that case (see \cite{haefliger_deformations_1983}).\\
Also note that a germ of deformation of a holomorphic foliation on a complex manifold parametrized by an analytic space $(S,0)$ induces naturally a germ of deformation of its associated transversally holomorphic foliation parametrized by the same analytic space $(S,0)$.
\begin{proposition}
    \label{prop:2.6}
    Let $\mathcal{F}^S$ a germ of deformation of a holomorphic foliation $\mathcal{F}$ on a complex manifold $M$ parametrized by analytic space $(S,0)$. Denote by $\rho^{tr}:T_0S\to H^1(M,\theta^{tr}_\mathcal{F})$ the Kodaira-Spencer map measuring the deformation of the induced transversally holomorphic foliation.\\
    Then $\rho^{tr}=p \circ \rho$ where $p : H^1(M,\theta_\mathcal{F}) \to H^1(M,\theta^{tr}_\mathcal{F})$ is the map induced by the projection $\theta_\mathcal{F} \to \theta^{tr}_\mathcal{F}$.
\end{proposition}
\begin{proof}
    It is a mere verification using the definitions. Consider an open covering $(U_i)_i$ and a holomorphic family of adapted holomorphic atlases $\psi_i^s$ defining the germ of deformation $\mathcal{F}^S$. Then the family of holomorphic submersions $(\text{pr}_2 \circ \psi_i^s)_i$ holomorphic in $s \in S$ along with the biholomorphic maps $h_{ij}^s=\text{pr}_2\circ (\psi_{i}^s\circ (\psi_{j}^s)^{-1})$ holomorphic in $s\in S$ define a germ of deformation of the induced transversely holomorphic foliation. The result follows by definition of the Kodaira-Spencer maps.
\end{proof}
\subsection{Infinitesimal deformation induced by a germ of deformation of a holomorphic vector field}
We consider the example of a deformation of a nowhere vanishing holomorphic vector field $\xi$ on a complex manifold $M$. Denote by $\mathcal{F}$ the holomorphic foliation induced by $\xi$, that is whose leaves are the orbits of $\xi$.\\
Throughout the following, denote by :
\begin{itemize}
    \item $\sigma$ the sheaf of holomorphic functions on $M$ ;
    \item $\theta$ the sheaf of holomorphic vector fields on $M$ ;
    \item $\theta^\xi$ the subsheaf of $\theta$ of vector fields commuting with $\xi$ (i.e. whose Lie derivative along the direction of $\xi$ is zero).
\end{itemize}
Consider the following morphisms of sheaves on $M$ : $L_\xi : \sigma \to \sigma$ is the Lie derivative on functions with respect to $\xi$, defined by $L_\xi(f)=\xi(f)$ ;  $L_\xi : \theta \to \theta$ is the Lie derivative on vector fields with respect to $\xi$, defined by $L_\xi(X)=[\xi, X]$ ; and  $m_\xi : \sigma \to \theta$ is the morphism defined by $m_\xi(f)=f \xi$.
\begin{lemma}
    The morphism $m_\xi$ restricts to a morphism $\sigma^{tr}_\mathcal{F} \to \theta^\xi$, which we still denote by $m_\xi$. Moreover, it makes the following diagram of sheaves commute :
    \[\begin{tikzcd}
        \sigma^{tr}_\mathcal{F} \arrow[r] \arrow[d] & \sigma \arrow[r, "L_\xi"] \arrow[d]& \sigma \arrow[d]\\
        \theta^{\xi} \arrow[r] & \theta \arrow[r, "L_\xi"] & \theta 
    \end{tikzcd}
    \]
    where the first horizontal arrows of each line are the inclusion morphisms.
\end{lemma}
\begin{proof}
    The commutativity of the first part of the diagram is immediate. 
    As for the second, let $U$ an open subset of $W$ and $f \in \sigma(U)$. Then (see \cite{lee_manifolds_2009}),
    $$L_\xi(f\xi)=\xi(f)\xi + f[\xi,\xi]=\xi(f)\xi$$
    which proves the result.
\end{proof}
\begin{proposition}
    \label{prop:2.8}
    The following complexes of sheaves are exact :
    \[\begin{tikzcd}[column sep = 5mm, row sep = 3mm]
        0 \arrow[r] & \sigma^{tr}_\mathcal{F} \arrow[r] & \sigma \arrow[r, "L_\xi"] & \sigma \arrow[r] & 0 \\
        0 \arrow[r] & \theta^{\xi} \arrow[r] & \theta \arrow[r, "L_\xi"] & \theta \arrow[r] & 0 \\
        0 \arrow[r] & \sigma^{tr}_\mathcal{F} \arrow[r, "m_\xi"] & \theta^\xi \arrow[r] & \theta^{tr}_\mathcal{F} \arrow[r] & 0 
    \end{tikzcd}
    \]
\end{proposition}
\begin{proof}
    The morphism $m_\xi$ as well as both the inclusions are obviously injective.
    Fix $x \in M$ and consider an open set $U_i$ of a holomorphic distinguished chart $$\psi_i: \left \{\begin{array}{ccl}
         U_i &\to& B_\mathbb C(0,1)\times B_{\mathbb C^{n-1}}(0,1)  \\
         p & \mapsto & (w,(z_1, \cdots, z_n))
    \end{array} \right .$$
    around $x$ for the holomorphic flow $\xi$ such that $\xi=\psi_i^*(\tfrac{\partial}{\partial w})$. 
    We prove the exactness on stalks, so we can assume $W$ is $U_i$ and $\xi=\dfrac{\partial}{\partial w}$.\\
    As for the first line : 
    \begin{itemize}
        \item a stalk $f_x \in \sigma_x$ is in $(\sigma^{tr}_{\mathcal{F}})_x$ if and only if $\left(\dfrac{\partial f}{\partial w} \right)_x=0$ if and only if $L_\xi(f_x)=0$ ;
        \item Let $f_x \in \sigma_x$ and $F_x \in \sigma_x$ whose holomorphic derivative with respect to $w$ in a simply connected open neighborhood of $x$ is $f$ (which exists by Goursat's theorem). Then $L_\xi(F_x)=\left(\dfrac{\partial F}{\partial w} \right)_x=f_x$.
    \end{itemize}
    As for the second line : 
    \begin{itemize}
        \item a stalk $Z_x \in \theta_x$ is in $(\theta^\xi)_x$ if and only if $L_\xi(Z_x)=0$ ;
        \item Let $Z_x=Z^{(0)}_x\dfrac{\partial}{\partial w}_x +\sum_{j=1}^{n-1} Z^{(j)}_x\dfrac{\partial}{\partial z_j}_x \in \theta_x$, where each $Z^{(j)}_x$ belongs to $\sigma_x$. By the previous point, for each $j \in \llbracket 0,n-1 \rrbracket$, we can find $X^{(j)}_x \in \sigma_x$ such that
        $$\left(\dfrac{\partial X^{(j)}}{\partial w} \right)_x=Z^{(j)}_x.$$ Therefore, if we note  $X_x=X^{(0)}_x\dfrac{\partial}{\partial w}_x +\sum_{j=1}^{n-1} X^{(j)}_x\dfrac{\partial}{\partial z_j}_x$, then
        $L_\xi(X_x)=Z_x$.
    \end{itemize}
    As for the third line : 
    \begin{itemize}
        \item a stalk $Z_x \in \theta^\xi_x$ is equal to the stalk at $x$ of a holomorphic vector field tangent to the leaves of $\mathcal{F}$ if and only if $Z_x=f_x\xi_x$ where $f\in \sigma_x$ is such that $\xi(f)=0$, i.e. $Z_x \in \text{Im}((m_\xi)_x)$;
        \item the last quotient map is obviously surjective.
    \end{itemize}
\end{proof}
\begin{remark}
    The above proves that for a distinguished chart $(U_i, f_i)$, 
        $$\restr{\theta^{\xi}}{U_i}\cong f_i^{-1}\left((\sigma_{\mathbb C^{n-1}})^{\oplus n} \right)\cong(f_i^{-1}\sigma_{\mathbb C^{n-1}})^{\oplus n}=\restr{(\sigma^{tr}_{\mathcal{F}})^{\oplus n}}{U_i}.$$
\end{remark}
These short exact sequences of sheaves give rise to long exact sequences of sheaf cohomology groups (see \cite{kashiwara_sheaves_1994} for example), which can be represented in the following commutative diagram (the first line is the long exact sequence associated to the first short exact sequence of sheaves; the second line is the long exact sequence associated to the second short exact sequence of sheaves ; the middle column is the long exact sequence associated to the third short exact sequence of sheaves; the vertical arrows from the first line to the second correspond to the maps induced by $m_\xi$) :

\[\begin{tikzcd}[column sep=5.87mm]
        & & & \vdots \arrow[d] & & & \\
        \cdots \arrow[r] & H^0(M, \sigma) \arrow[r,"L_\xi"] \arrow[d] & H^0(M,\sigma) \arrow[r, "\delta"] \arrow[d] & H^1(M,\sigma^{tr}_\mathcal{F}) \arrow[r] \arrow[d,"m_\xi"] & H^1(M,\sigma) \arrow[r,"L_\xi"] \arrow[d] & H^1(M,\sigma) \arrow[r] \arrow[d]  &\cdots \\
        \cdots \arrow[r]
 &  H^0(M, \theta)  \arrow[r,"L_\xi"] &  H^0(M,\theta) \arrow[r, "\delta'"] \arrow[rd, swap, "p\circ\delta' "] & H^1(M,\theta^\xi) \arrow[r] \arrow[d, "p"] & H^1(M,\theta) \arrow[r,"L_\xi"] & H^1(M,\theta) \arrow[r] & \cdots\\
  & &  &  H^1(M, \theta^{tr}_\mathcal{F}) \arrow[d]  &  &  & \\
  & &  &  H^2(M, \sigma^{tr}_\mathcal{F}) \arrow[d]  &  &  & \\
  & &  &  \vdots &  &  & 
    \end{tikzcd}
    \]

Let $\xi^s$ a holomorphic family of nowhere vanishing holomorphic vector fields on $M$ parametrized by a germ of analytic space $(S,0)$ such that $\xi^0=\xi$. Denote by $\mathcal{F}$ the foliation whose leaves are the orbits of $\xi$ and by $\mathcal{F}^S$ the germ of deformation of $\mathcal{F}$ parametrized by $(S,0)$ induced by $\xi^s$.
\begin{proposition}
    \label{prop:2.9}
    The Kodaira-Spencer map $\rho : T_0S 
    \to H^1(M, \theta_\mathcal{F})$ measuring the infinitesimal deformation of $\mathcal{F}^S$ is given, for $\dfrac{\partial}{\partial s} \in T_0S$, by :
    $$\rho \left(\dfrac{\partial}{\partial s} \right)=-\mathcal{i}\circ \delta'\left ( \restr{\dfrac{\partial \xi^s}{\partial s}}{s=0}\right )$$
    where $\mathcal{i} : H^1(M, \theta^\xi) \to H^1(M,\theta_\mathcal{F})$ is the map induced by the inclusion of $\theta^\xi$ in $\theta_\mathcal{F}$.
\end{proposition}
\begin{proof}
    By definition, $\rho(\restr{\dfrac{\partial}{\partial s}}{0})$ is the cohomology class of the cocycle associating to $U_i \cap U_j$ the section $$\theta_{ij}:=(d\psi_i^0)^{-1} \circ\left (\restr{\dfrac{\partial g_{ij}^s}{\partial s}}{0} \right )\circ \psi_j^0=\restr{\dfrac{\partial}{\partial s}}{0} \left((\psi_i^0)^{-1} \circ g_{ij}^s \circ \psi_j^0 \right )$$ of $\theta_\mathcal{F}(U_i \cap U_j)$. Let $\eta_i$ the holomorphic vector field on $U_i$ defined by 
    $$\eta_i=\restr{\dfrac{\partial}{\partial s}}{0} \left ((\psi_i^s)^{-1}\circ \psi_i^0 \right).$$
    The equality  $\psi_i^s=g_{ij}^s\circ \psi_j^s$ on $U_i \cap U_j$ implies $$(\psi_j^s)^{-1}\circ \psi_j^0=\left((\psi_i^s)^{-1}\circ \psi_i^0 \right)\circ\left ( (\psi_i^0)^{-1}\circ g_{ij}^s  \circ \psi_j^0\right )$$
    on $((\psi_j^0)^{-1}\circ\psi_j^s)(U_i \cap U_j)$, and thus after applying the derivation $\restr{\dfrac{\partial}{\partial s}}{0}$ : $$\theta_{ij}=\eta_j-\eta_i \quad \text{ on }U_i\cap U_j.$$
    On the other hand, $\xi^s=((\psi_i^s)^{-1})_*\left (\dfrac{\partial}{\partial z} \right)=((\psi_i^s)^{-1}\circ \psi_i^0)_*(\xi)$ on $U_i$, so applying the derivation $\restr{\dfrac{\partial}{\partial s}}{0}$ to this equality gives :
    $$\restr{\dfrac{\partial \xi^s}{\partial s}}{s=0}=[\eta_i,\xi]=L_\xi(-\eta_i).$$
    The results follows by definition of $\delta'$.
\end{proof}
By the previous result and Proposition \ref{prop:2.6}, it comes :
\begin{corollary}
    The Kodaira-Spencer map $\rho : T_0S 
\to H^1(M, \theta^{tr}_\mathcal{F})$ measuring the germ of deformation of the transversely holomorphic foliation induced by $\mathcal{F}^S$ is given, for $\dfrac{\partial}{\partial s} \in T_0S$, by :
$$\rho \left(\dfrac{\partial}{\partial s} \right)=-p\circ \delta'\left ( \restr{\dfrac{\partial \xi^s}{\partial s}}{s=0}\right ).$$
\end{corollary}
\section{Simplification of the problem}
\label{sec:3}
Consider a holomorphic flow $\xi$ defined in a neighborhood of a neighborhood of the closure of a bounded strongly convex domain $\Omega \subset\mathbb C^n$ ($n\geq 2$) with smooth boundary (we refer to \cite{krantz_function_2001} for notions of convexity for domains of $\mathbb R^N$). 
Examples of such $\Omega$ can be the open unit ball $B^n$ or more generally $B^n_p:=\left \{ (z_1,\cdots, z_n)\in\mathbb C^n\setminus0, \; \sum_{k=1}^n|z_k|^p < 1 \right \}$
for $p\geq 2$. 

Since the aim is to find a holomorphic diffeomorphism $\Phi$ on a neighborhood of $\overline{\Omega}$ taking value in a neighborhood of $\overline{B^n} \subset \mathbb C^n$, sending $\xi$ to $\Phi_*\xi$ and satisfying certain properties, we will simplify notations and still denote by $\xi$ the newly obtained holomorphic vector field, without mentioning $\Phi$. Also, because the goal is to find a versal deformation of $\mathcal{F}_0(\xi)$ on $\partial \Omega$, the previous remarks allow us to send $\mathcal{F}_0(\xi)$ by a diffeomorphism $h$ and consider the corresponding transversely holomorphic foliation $h_*(\mathcal{F}_0(\xi))$. In case the diffeomorphism $h$ is given by the restriction of a holomorphic diffeomorphism $\Phi$ to the boundary of a set $M'$, then $h_*(\mathcal{F}_0(\xi))$ is just the transversely holomorphic foliation on $\Phi(\partial M')$ obtained by intersecting the leaves of the orbit foliation of $\Phi_*\xi$ with $\Phi(\partial M')$.

Having said that, we can assume $\xi$ is a holomorphic vector field defined on  a neighborhood of the closure of a bounded strongly convex domain $\Omega \subset\mathbb C^n$ with smooth boundary, which is moreover transverse to the boundary $\partial \Omega$.

In \cite{brunella_remarque_1994}, Marco Brunella proved the following :

\begin{mythm}{\cite{brunella_remarque_1994}}
    Let $\Omega$ be a bounded strongly convex domain of $\mathbb C^n$ with smooth boundary $\partial \Omega$, where $n\geq 2$. Let $\xi$ a holomorphic vector field defined on a neighborhood of $\overline{\Omega}$ which is transverse to $\partial \Omega$.\\
    Then there exists a biholomorphism $\Phi$ from a neighborhood of $\overline{\Omega}$ to a neighborhood of $\overline{B^n}$ such that $\Phi_*\xi$ has a unique singularity $p$ in $\Omega$ and that $\Phi_*\xi$ is transverse to the spheres $S^{2n-1}(r):=\left \{\right (z_1,\cdots, z_n)\in\mathbb C^n\setminus0, \; \sum_{k=1}^n|z_k|^2=r\}$ for each $r\in (0,1]$.
\end{mythm}

In the meantime, assuming $\Omega=B^n$, Ito showed the same result in \cite{ito_poincare-bendixson_1994} :
\begin{mythm}{1 \cite{ito_poincare-bendixson_1994}}
    Let $M$ be a subset of $\mathbb C^n$ diffeomorphic to the closed unit ball $\overline{B^n}\subset \mathbb C^n$. Let $Z$ a holomorphic vector field defined on a neighborhood of $M$ and transverse to the boundary of $M$.\\
    Then $Z$ admits a unique singularity $p$ in $M$. 
    Furthermore, the index of $Z$ at $p$ is one.
\end{mythm}
\begin{mythm}{3 \cite{ito_poincare-bendixson_1994}}
    Let $M$ be a subset of $\mathbb C^n$ biholomorphic to the closed unit ball $\overline{B^n}\subset \mathbb C^n$. Let $Z$ a holomorphic vector field defined on a neighborhood of $M$ and transverse to the boundary of $M$.\\
    Then each leaf $L$ of the foliation induced by $\xi$ converges to $p$, i.e. $p$ is in the closure of $L$. Also the restriction on $\overline{B^n}\setminus p$ of the foliation $\mathcal{F}$ induced by $\xi$ is $C^w$-isomorphic to the foliation $\mathcal{F}_0 \times (0,1]$ on $M\setminus p$.
\end{mythm}

More precisely, by using a Möbius transformation sending $p$ to $0$ (see \cite{rudin_function_2008}), he restricts to the case where $p=0$. He then proved that $\xi$ is transverse to the sphere $S^{2n-1}(r):=\left \{\right (z_1,\cdots, z_n)\in\mathbb C^n\setminus0, \; \sum_{k=1}^n|z_k|^2=r\}$ for each $r\in (0,1]$. 

Therefore, we can assume from the beginning that $\Omega=B^n$, that $\xi$ has a unique singularity in $B^n$ which is $0\in \mathbb C^n$, and that $\xi$ is transverse to the spheres $S^{2n-1}(r)$ for each $r\in (0,1]$.

The latter means the following for $r=1$ (see \cite{ito_poincare-bendixson_1994}). Write for $z=(x,y) \in \mathbb C^n \cong \mathbb R^{2n}$, $$\xi=\sum_{j=1}^n F_j(z)\frac{\partial}{\partial z_j}$$ where each $F_j(z)=g_j(x,y)+ih_j(x,y)$ is a holomorphic map from $\Omega$ to $\mathbb C$.
Then
\begin{align*}
    \xi&=\sum_{j=1}^n F_j(z)\frac{\partial}{\partial z_j}=\sum_{j=1}^n (g_j(x,y)+ih_j(x,y))\cdot \frac{1}{2}\left(\frac{\partial}{\partial x_j} - i\frac{\partial}{\partial y_j} \right)\\
    &=\frac{1}{2}\left (\sum_{j=1}^n \left (g_j(x,y)\frac{\partial}{\partial x_j} + h_j(x,y)\frac{\partial}{\partial y_j} \right ) - i \sum_{j=1}^n \left( -h_j(x,y)\frac{\partial}{\partial x_j} + g_j(x,y)\frac{\partial}{\partial y_j}\right )\right) = \frac{1}{2}\left(X-iY \right ),
\end{align*}
where
$$X=\sum_{j=1}^n \left (g_j(x,y)\frac{\partial}{\partial x_j} + h_j(x,y)\frac{\partial}{\partial y_j} \right) \quad \text{ and } \quad Y=\sum_{j=1}^n \left (-h_j(x,y)\frac{\partial}{\partial x_j} + g_j(x,y)\frac{\partial}{\partial y_j}\right ).$$
Therefore, the fact that $\xi$ is transverse to $S^{2n-1} \subset \mathbb C^n$ means that for every $p \in S^{2n-1}$ :
$$T_pS^{2n-1} + \text{Vect}_\mathbb R\left(X(p),Y(p)\right)=T_p \mathbb R^{2n},$$
i.e. $X(p)$ and $Y(p)$ do not belong to the tangent space at $p$ of $S^{2n-1}$.
If we note $$\mathbf N=\sum_{j=1}^n\left( x_j\frac{\partial}{\partial x_j}+y_j\frac{\partial}{\partial y_j} \right)$$ the usual vector field on $\mathbb R^{2n}$ normal to the spheres, the tangency of a vector field $W$ at $p$ to $S^{2n-1}$ means exactly that $W(p)$ is orthogonal (for the usual inner product $\langle \cdot, \cdot \rangle$ on $\mathbb R^{2n}$) to $\mathbf N(p)$.
As a result, $\xi$ is transverse to $S^{2n-1}$ if and only if for every $p\in S^{2n-1}$, 
$\sum_{j=1}^n F_j(p)\overline{z_j}\neq 0$
since $\sum_{j=1}^n F_j(p)\overline{z_j}=\langle X,\mathbf{N}\rangle - i\langle Y, \mathbf N \rangle $, that is $\xi(p)$ is not orthogonal to $p$ for the usual hermitian product on $\mathbb C^n$.\\
For a holomorphic diagonal vector field $\xi_0=\sum_{j=1}^n\lambda_jz_j\frac{\partial}{\partial z_j}$, the latter is equivalent to the fact that $(\lambda_1,\cdots, \lambda_n) \in \mathbb C^n$ does not belong to the Poincaré domain, i.e. $0\in \mathbb C^n$ does not belong to the convex hull generated by $\{\lambda_1, \cdots, \lambda_n \}$. In \cite{arnold_remarks_1969}, Arnold proved the following result :
\begin{proposition}
    Let $Z$ a holomorphic vector field defined on a neighborhood of $0\in \mathbb C^n$. \\
    If the set of eigenvalues of the differential at $0$ of $Z$ belongs to the Poincaré domain, then there exists $r_0>0$ such that for every $0<r \leq r_0$, $Z$ is transverse to $S^{2n-1}(r)$ the sphere of radius $r$.
\end{proposition}
Since $\xi$ is transverse to the sphere $S^{2n-1}$, it does not vanish on a neighborhood $U$ of $S^{2n-1}$. Hence, it defines a holomorphic foliation $\mathcal{F}$ of complex dimension one (or of complex codimension $n-1$) on $U$ whose leaves are the orbits of $\xi$ (see \cite{camacho_geometric_1985} for example in the smooth case), that is : there exists an open covering $\mathcal{U}$ of $U$ by open sets $U_i$ which intersect $S^{2n-1}$ (if we shrink $U$) as well as holomorphic submersions $f_i: U_i \to \mathbb C^{n-1}$ and biholomorphic maps $g_{ij} : f_j(U_i\cap U_j) \to f_i(U_i\cap U_j)$ satisfying :
$$f_i=g_{ij}\circ f_j \quad \text{ on } U_i\cap U_j.$$
We can use this foliated atlas of $U$ to define a transversely holomorphic foliation $\mathcal{F}_0$ of complex codimension $n-1$ on $S^{2n-1}$ as follows. The sets $(S^{2n-1}\cap U_i)_i$ form an open cover of the sphere, on which we can define a smooth map $s_i:S^{2n-1}\cap U_i\to \mathbb C^{n-1}$ by 
$$s_i=f_i\circ \iota$$
where $\iota: S^{2n-1} \to \mathbb C^{n}\setminus{0}$ is the inclusion. For $p\in S^{2n-1}\cap U_i$, since the differential at $p$ of $f_i$ vanish along $\xi(p)$, $\xi(p)$ is transverse to $S^{2n-1}$ and $f_i$ is a submersion, it comes that 
$$d_ps_i\left(T_p(S^{2n-1}\cap U_i)\right)=d_pf_i\left(T_pS^{2n-1}\right)=d_pf_i\left(T_p\mathbb C^n\right)=d_pf_i\left(T_pU_i\right)=T_{s_i(p)}\mathbb C^{n-1},$$
which means that $s_i:S^{2n-1}\cap U_i\to \mathbb C^{n-1}$ is a submersion. Moreover, $$s_i=\restr{g_{ij}}{f_j(S^{2n-1}\cap U_i\cap U_j)}\circ s_j \quad \text{ on } S^{2n-1}\cap U_i\cap U_j$$
and $\restr{g_{ij}}{f_i(S^{2n-1}\cap U_i\cap U_j)}$ is holomorphic on the open set $$f_j(S^{2n-1}\cap U_i \cap U_j)=s_j(S^{2n-1}\cap U_i \cap U_j)\subset f_j(U_i \cap U_j)$$
since $s_i$ is a submersion thus an open map. Its holomorphic inverse is obviously the restriction of $g_{ji}$ to $f_i(S^{2n-1}\cap U_i \cap U_j)$.\\
Remark that the leaves of the foliation $\mathcal{F}_0$ can be oriented, in the following way. Fix $p \in S^{2n-1}$. Consider a unit vector $v(p) \in T_p\mathcal{F}$ tangent to the leaf of $\mathcal{F}$ at $p$ such that $(v(p), \widetilde{\mathbf N}(p))$ defines the positive orientation of the complex structure of $\mathcal{F}_p$, where $\widetilde{\mathbf {N}}(p)$ is the orthogonal projection of $\mathbf{N}(p)$ on $T_p\mathcal{F}$. The map $p\mapsto v(p)$ defines a non-vanishing smooth vector field on $S^{2n-1}$ whose orbits are the leaves of $\mathcal{F}_0$.\\
Also, remark that a germ of deformation of $\mathcal{F}$ parametrized by $(S,0)$ induces naturally a germ of deformation of $\mathcal{F}_0$ parametrized by $(S,0)$.

In \cite{ito_number_1996}, Ito proved that the set of eigenvalues $(\lambda_1, \cdots, \lambda_n) \in \mathbb C^n$ of the differential at $0$ of $\xi$ must belong to the Poincaré domain.\\
Now recall the Poincaré-Dulac theorem (see \cite{szucs_geometrical_1996}):
\begin{theorem}[Poincaré-Dulac]
    Let $Z$ a holomorphic vector field defined on a neighborhood of $0\in \mathbb C^n$.\\
    If the family $(\lambda_1, \cdots, \lambda_n)$ of eigenvalues (counted with multiplicity) of the differential at $0$ of $Z$ belongs to the Poincaré domain, then there exists a biholomorphic map $\Phi$ defined on a neighborhood $U$ of $0$ in $\mathbb C^n$ satisfying :
    \begin{itemize}
        \item $\Phi(0)=0$ ;
        \item If we note $Y:=\Phi_*Z$ the push-forward of $Z$ by $\Phi$, then
        $$Y=\sum_{j=1}^n\lambda_jw_j \frac{\partial}{\partial w_j}+\sum_{j=2}^n \left(b_jw_{j-1}+P_j(w_1, \cdots, w_{j-1}) \right) \frac{\partial}{\partial w_j}$$
        where each $b_j$ is either $0$ or $1$ and can be seen on the Jordan blocks of $d_0Z\in \mathrm{M}_n(\mathbb C)$, and each $P_j$ is a polynomial defined as 
        $$P_{j}(w_1, \cdots, w_{j-1})=\sum_{m_j}a_{m_j}w^{m_j}$$
        where the sum is over the set of $(j-1)$-tuples $m_j=(m_j^{(1)}, \cdots, m_j^{(j-1)}) \in \mathbb N^{j-1}$ such that $$|m_j|:=\sum_{k=1}^{j-1} m_j^{(k)} \geq 2  \quad \text{ and } \quad \lambda_j= \sum_{k=1}^{j-1}m_j^{(k)}\lambda_k,$$
        $w^{m_j}:=w_1^{m_j^{(1)}}\cdots \; w_{j-1}^{m_j^{(j-1)}}$ and each $a_{m_j} \in \mathbb C$ is determined by $Z$.
    \end{itemize}
\end{theorem}
\begin{remark}
    $Y$ can be written more compactly as $$Y=\sum_{j=1}^n\lambda_jw_j \frac{\partial}{\partial w_j}+\sum_{(j,m)\in R} a_{j,m}w^m\frac{\partial}{\partial w_j}$$
    where the last sum is over the finite set (see Proposition \ref{prop:4.3})
    \begin{align*}
        R &=\left \{(j,(m_1, \cdots, m_n)) \in \llbracket 2, n \rrbracket \times \mathbb N^n, \; |m|\geq 2, \quad \forall k \geq j, m_k=0, \quad\lambda_j=\sum_{k=1}^n m_k\lambda_k\right \} \\
        &\bigcup\left \{(j,e_{j-1}) \in \llbracket 2, n \rrbracket \times \mathbb N^n\right \}
    \end{align*}
    and each $a_{j,m}$ is a complex number. Recall that, for $k\in \llbracket1,n \rrbracket$, $e_k \in \mathbb N^n$ is the $n$-tuple whose entries are zeroes except the $k$-th one which equals $1$.
\end{remark}
By the grace of this result and the fact that the foliation $\mathcal{F}_0$ of $S^{2n-1}$ is $C^w$-isomorphic to the foliation $\restr{\mathcal{F}}{S^{2n-1}(r)}$ of $S^{2n-1}(r)$ for $r\in (0,1]$ as small as we want by the correspondence along the orbits of $\widetilde{\mathbf N}$, we can assume that $\xi$ writes as $$\xi=\sum_{j=1}^n\lambda_jz_j \frac{\partial}{\partial z_j}+\sum_{(j,m)\in R} a_{j,m}z^m\frac{\partial}{\partial z_j}$$ and that $\mathcal{F}_0$ is given by the transverse intersection of $\mathcal{F}$ with the unit sphere $S^{2n-1}$.
Also, as mentioned in \cite{haefliger_deformations_1985},  we can assume that the coefficients $a_{j,m}$ are as small as we want. Indeed, if $A>1$ is large enough and $h$ is the diagonal linear map whose entries are $A^{1}, A^{1/2}, \cdots, A^{1/n}$ in that order, then 
$$h_*\xi=\sum_{j=1}^n\lambda_jz_j \frac{\partial}{\partial z_j}+\sum_{(j,m)\in R} \frac{A^{1/j}}{A^{m_1+m_2/2+\cdots +m_{j-1}/(j-1)}}a_{j,m}z^m\frac{\partial}{\partial z_j},$$
therefore, since for any $(j,m) \in R$ there exists $k\in \llbracket1,j-1 \rrbracket$ such that $m_k\geq 1$, it comes that 
$$\frac{A^{1/j}}{A^{m_1+m_2/2+\cdots +m_{j-1}/(j-1)}} \leq A^{1/j-1/k}=A^{-\tfrac{j-k}{jk}}$$
which concludes.\\
In that case, since $\xi_0:=\sum_{j=1}^n\lambda_jz_j \frac{\partial}{\partial z_j}$ is transverse to $S^{2n-1}$, we can take the coefficients $a_{j,m}$ small enough so that $\xi$ is transverse to $S^{2n-1}$ and thus to $S^{2n-1}(r)$ for every $r \in (0,1]$.
As a result, by the correspondence along the orbits of $\widetilde{N}$, we can assume $\mathcal{F}_0$ is given by the intersection of $\mathcal{F}$ and $S^{2n-1}$.\\
To put it in a nutshell, we can restrict ourselves to the case where :
\begin{itemize}
    \item $\xi$ is equal to the holomorphic polynomial vector field $$\sum_{j=1}^n\lambda_jz_j \frac{\partial}{\partial z_j}+\sum_{(j,m)\in R} a_{j,m}z^m\frac{\partial}{\partial z_j}; $$
    \item the coefficients $a_{j,m}$ can be taken as small as we want ;
    \item $\xi$ is transverse to  $S^{2n-1}(r)$ for every $r\in (0,1]$ ;
    \item $\mathcal{F}_0$ is given by the intersection of $\mathcal{F}(\xi)$ and $S^{2n-1}$.
\end{itemize}
Moreover, the set of eigenvalues $(\lambda_1, \cdots, \lambda_n) \in \mathbb C^n$ of the differential at $0$ of $\xi$ belongs to the Poincaré domain.
\begin{proposition}
\label{prop:3.3}
    The vector field $\xi = \sum_{j=1}^n\lambda_jz_j \dfrac{\partial}{\partial z_j}+\sum_{(j,m)\in R} a_{j,m}z^m\dfrac{\partial}{\partial z_j}$ is a non-vanishing holomorphic vector field on $W:=\mathbb C^n \setminus0$ whose holomorphic flow is given, for $j \in \llbracket1,n\rrbracket$ and $t \in \mathbb C$, by :
    \begin{align*}
        z_j(t)=e^{\lambda_st}\left( z_j(0)+ \sum_{r=1}^\infty\left(\sum_{(j,m)\in R} b_{j,m}^{(r)}z(0)^m \right)t^r \right)
    \end{align*}
    where each $b_{j,m}^{(r)}$ is a complex number.
\end{proposition}
\begin{proof}
    Since, for $j\in \llbracket2,n\rrbracket$, the $j$-th component of $\xi_R:=\sum_{(j,m)\in R} a_{j,m}z^m\dfrac{\partial}{\partial z_j}$ depends only on the holomorphic functions $z_1, \ldots, z_{j-1}$, and because the first component of $\xi$ is $\lambda_1z_1\dfrac{\partial}{\partial z_1}$, we can solve $z_1$ ($z_1(t)=e^{\lambda_1t}z_1(0)$) and then $z_j$ for any $j\in \llbracket2,n\rrbracket$ by induction so that it has the desired form. Indeed, assume the result true for a fixed $j\in \llbracket1,n-1\rrbracket$. Then, for $t \in \mathbb C$, 
    $$z_{j+1}'(t)=\lambda_{j+1}z_{j+1}(t)+f(t)$$
    where $$f(t)=\sum_{(j,m)\in R} a_{j,m}(z(t))^m.$$
    We can write :
    $$\left (z_{j+1}e^{-\lambda_{j+1}t}\right)'(t)=e^{-\lambda_{j+1}t}f(t)$$
    which proves the result by induction hypothesis since, by definition of $R$, $f(t)$ depends only on $z_1(t), \ldots, z_{j-1}(t)$. \\
    If $\xi(z)=0$, the same argument first gives $z_1=0$ by looking at the first component of $\xi$, since $\lambda_1\neq 0$, and then $z_j=0$ for every $j\in \llbracket1,n\rrbracket$.
\end{proof}
\section{Restatement of the theorem} 
\label{sec:4}
\begin{definition}
As defined by Arnold in \cite{arnold_chapitres_1996}, a sequence of complex numbers $\lambda=(\lambda_1, \ldots , \lambda_n) \in \mathbb C^n$ is \textit{resonant} if there exists a relationship of the type
$$\lambda_s=(m,\lambda):=\sum_{k=1}^nm_k\lambda_k$$
where $m=(m_1, \ldots, m_n) \in \mathbb N^n$ satisfies $|m|:=\sum_{k=1}^n m_k \geq 2$.
Such relationship is called a resonance. We will also consider resonances with $|m|=1$ and call trivial the resonances of the form $\lambda_s=\lambda_{s}$.
\end{definition}
Remark that in our case where $\lambda=(\lambda_1, \ldots , \lambda_n)$ belongs to the Poincaré domain, such a resonance must satisfy $m\neq 0$.
\begin{definition}
Given a sequence of complex numbers $\lambda=(\lambda_1, \ldots , \lambda_n)$, a monomial vector field in $\mathbb C^n$ of the form $$az^m\cfrac{\partial}{\partial z_s}$$ where $a \in \mathbb C$, $s\in \llbracket 1 , n \rrbracket$, , $m=(m_1, \ldots, m_n) \in \mathbb N^n$ and $z^m:={z_1}^{m_1}\ldots {z_n}^{m_n}$, is called \textit{$\lambda$-resonant} if $\lambda_s=(m,\lambda)$.
\end{definition}
In our case where $\lambda=(\lambda_1, \cdots, \lambda_n)$ is in the Poincaré domain, there cannot be infinitely many resonances :
\begin{proposition}
    \label{prop:4.3}
    Every point $\lambda=(\lambda_1, \cdots, \lambda_n)$ of the Poincaré domain satisfies only a finite number of resonances $\lambda_s=(m,\lambda)$, $m\in \mathbb N^n$
\end{proposition}
\begin{proof}
    Necessarily, there exists $C>0$ such that every resonance $\lambda_s=(m,\lambda)$ satisfy $|m|\leq C$. Otherwise, there exists a sequence $(m^{(N)})_{N \in \mathbb N} \in (\mathbb N^n)^\mathbb N$, with $|m^{(N)}| \geq N+1$ for every $N \in \mathbb N$, and a resonance $\lambda_{s_N}=(m^{(N)}, \lambda)$. 
    Dividing this equality by $|m^{(N)}|$ for $N \in \mathbb N$ and letting $N$ go to infinity leads to $0$ belonging to the (closed) convex hull of $\{ \lambda_1, \cdots, \lambda_n\}$ which is absurd. Therefore, there must be finitely many $m=(m_1, \cdots, m_n)$ for each $\lambda_s$ since each $m_j$ is a non-negative integer.
\end{proof}
\begin{definition}
    We note $\mathcal{g}_\lambda$ the complex vector space of holomorphic vector fields on $\mathbb C^n$ which are (finite) sum of $\lambda$-resonant monomial vector fields and by $\mathcal{g}_\lambda^\perp$ the complex vector space of holomorphic vector fields on $\mathbb C^n$ which are (infinite) sum of non-$\lambda$-resonant monomial vector fields.
\end{definition}
The vector field $\xi=\sum_{j=1}^n\lambda_jz_j \tfrac{\partial}{\partial z_j}+\sum_{(j,m)\in R} a_{j,m}z^m\dfrac{\partial}{\partial z_j}$ belongs to $\mathcal{g}_\lambda$ since, as we said, we consider trivial resonances also in $\mathcal{g}_\lambda$, and the first sum corresponds exactly to the monomial vector fields associated to the (trivial) resonances $\lambda_s=\lambda_s$.
\begin{proposition}
\label{prop:4.5}
\leavevmode
\begin{enumerate}[label=(\roman*)]
    \item The vector space $\mathcal{g}_\lambda$ of $\lambda$-resonant vector fields is exactly the set of holomorphic vector fields on $\mathbb C^n$ commuting with the diagonal vector field $\xi_0:=\sum_{j=1}^n\lambda_jz_j \dfrac{\partial}{\partial z_j}$ ;
    \item $\mathcal{g}_\lambda$ is a finite dimensional vector space and a Lie subalgebra of the Lie algebra of holomorphic vector fields on $\mathbb C^n$ ;
    \item The bracket of a monomial $\lambda$-resonant vector field with a monomial $\lambda$-resonant vector field is a sum of two monomial $\lambda$-resonant vector fields.\\
    The bracket of a monomial $\lambda$-resonant vector field with a monomial non-$\lambda$-resonant vector field is a sum of two monomial non-$\lambda$-resonant vector fields ;
    \item The Lie derivative $L_\xi$ along the direction of $\xi$ maps $\mathcal{g}_\lambda$ to $\mathcal{g}_\lambda$ and $\mathcal{g}_\lambda^\perp$ to $\mathcal{g}_\lambda^\perp$ ;
    \item If $\theta$ is the sheaf of holomorphic vector fields on $W= \mathbb C^n\setminus0$, then
    $$H^0(W, \theta)=\mathcal{g}_\lambda\oplus\mathcal{g}_\lambda ^\perp.$$
\end{enumerate}
\end{proposition}
     \begin{proof}
         $(i)$ : Let $X=\sum_{j=1}^n X_j \dfrac{\partial}{\partial z_j}$ a holomorphic vector field on $\mathbb C^n$. Write $X_j=\sum_{m\in \mathbb N^n} b_m^{(j)}z^m$ where each $a_m^{(j)}$ is a complex number.
         Then (see \cite{lee_manifolds_2009})\begin{align*}
             L_{\xi_0}(X)&=[\xi_0,X]=dX(\xi_0)-d\xi_0(X)=\sum_{j=1}^n\left(\sum_{k=1}^n \dfrac{\partial X_j}{\partial z_k}\lambda_k z_k -X_j\lambda_j  \right) \dfrac{\partial}{\partial z_j}\\
             &=\sum_{j=1}^n\left(\sum_{k=1}^n \sum_{m\in \mathbb N^n} b_m^{(j)}m_k\lambda_kz^m -\sum_{m\in \mathbb N^n} b_m^{(j)}\lambda_jz^m \right) \dfrac{\partial}{\partial z_j}\\
             &=\sum_{j=1}^n\left(\sum_{m\in \mathbb N^n} b_m^{(j)}\left( \sum_{k=1}^nm_k\lambda_k -\lambda_j \right)z^m \right) \dfrac{\partial}{\partial z_j}.
         \end{align*}
         Therefore, by properties of power series, $X=\sum_{j=1}^n X_j \dfrac{\partial}{\partial z_j}=\sum_{j=1}^n\left (\sum_{m\in \mathbb N^n} b_m^{(j)}z^m \right)\dfrac{\partial}{\partial z_j}$ commutes with $\xi_0=\sum_{j=1}^n\lambda_j z_j \dfrac{\partial}{\partial z_j}$ if and only if for every $j \in \llbracket1,n \rrbracket$ and $m \in \mathbb N^n$, $b_m^{(j)}\left( \sum_{k=1}^nm_k\lambda_k -\lambda_j\right)=0$, i.e. if and only if for every $j \in \llbracket1,n \rrbracket$ and $m \in \mathbb N^n$ such that $b_m^{(j)}$ is not zero, $ \lambda_j = \sum_{k=1}^nm_k\lambda_k$. This means exactly that $X\in \mathcal{g}_\lambda$.\\
         $(ii)$ : By the previous point and Proposition \ref{prop:4.3}, $\mathcal{g}_\lambda$ is a finite dimensional vector subspace of the space of holomorphic vector fields on $\mathbb C^n$. Also, recall the Jacob identity for smooth vector fields $X,Y,Z$ defined on an open subset of a smooth manifold : 
         $$[X,[Y,Z]]+[Y,[Z,X]]+[Z,[X,Y]]=0.$$
         Therefore, by the previous point and this identity, if $X$ and $Y$ belong to $\mathcal{g}_\lambda$, then $[\xi_0,[X,Y]]=0$ i.e. $[X,Y]\in \mathcal{g}_\lambda$ and $\mathcal{g}_\lambda$ is a Lie subalgebra of the algebra of holomorphic vector fields on $\mathbb C^n$ .\\
         $(iii)$ : We compute $[z^m\cfrac{\partial}{\partial z_k},z^l\dfrac{\partial}{\partial z_j}]=l_k \dfrac{z^{m+l}}{z_k}\dfrac{\partial}{\partial z_j}-m_j \dfrac{z^{m+l}}{z_j} \dfrac{\partial}{\partial z_k}$. Remark that if $l_k$ is non zero, then $\dfrac{z^{m+l}}{z_k}=z^{m+l-e_k}$ where $e_k \in \mathbb N^n$ is the $n$-tuple whose entries are zeroes except the $k$-th one which equals $1$.\\
         If $l_k \neq 0$, assume $\lambda_j=(l,\lambda)$. Then $\lambda_k+\lambda_j=(m+l,\lambda)$ thus $\lambda_j=(m+l-e_k,\lambda)$ and $l_k \dfrac{z^{m+l}}{z_k}\dfrac{\partial}{\partial z_j}$ is $\lambda$-resonant. It is also true in the trivial case where $l_k=0$. The same can be said for $m_j \dfrac{z^{m+l}}{z_j} \dfrac{\partial}{\partial z_k}$.\\
         If $l_k \neq 0$, assume $\lambda_j\neq(l,\lambda)$. Then $\lambda_j\neq(m+l-e_k,\lambda)$ so $l_k \dfrac{z^{m+l}}{z_k}\dfrac{\partial}{\partial z_j}$ is not $\lambda$-resonant (otherwise, we would have $\lambda_k+\lambda_j=(m+l,\lambda)$ thus $\lambda_j=(l,\lambda)$ since $\lambda_k=(m,\lambda)$). \\
         If $l_k=0$, $l_k \dfrac{z^{m+l}}{z_k}\dfrac{\partial}{\partial z_j}=0\in \mathcal{g}_\lambda^\perp$. The same can be said for $m_j \dfrac{z^{m+l}}{z_j} \dfrac{\partial}{\partial z_k}$.\\
         $(iv)$ : Since $\xi$ is a finite sum of $\lambda$-resonant vector fields, it suffices to show the result for a $\lambda$-resonant vector field $z^m\cfrac{\partial}{\partial z_k}$ instead of $\xi$. 
         First remark that if $Y$ is a holomorphic vector field on an open subset $U$ of $\mathbb C^n$ and if $f=\sum_{m\in \mathbb N^n} b_mz^m$ is a holomorphic function on $U$, then
         $$Y(f)=\sum_{m\in \mathbb N^n} b_mY(z^m).$$
         Indeed, this is true for $Y=\dfrac{\partial}{\partial z_j}$ by properties of power series, and since $Y$ is a linear combination on $U$ of $\dfrac{\partial}{\partial z_1}, \cdots, \dfrac{\partial}{\partial z_n}$, then it is also true for $Y$ by finite sum. As a result, if $X=\sum_{j=1}^n F_j \dfrac{\partial}{\partial z_j}$ is a holomorphic vector field on $U$, where $F_j=\sum_{m\in \mathbb N^n} b_m^{(j)}z^m$, it comes 
         \begin{align*}
             [Y,X]&=\sum_{j=1}^n[Y,F_j\dfrac{\partial}{\partial z_j}]=\sum_{j=1}^n \left (Y(F_j)\dfrac{\partial}{\partial z_j}+F_j[Y,\dfrac{\partial}{\partial z_j}] \right)\\
             &=\sum_{\substack{1\leq j \leq n\\m\in \mathbb N ^n}}b_m^{(j)}\left(Y(z^m)\dfrac{\partial}{\partial z_j}+ z^m[Y,\dfrac{\partial}{\partial z_j}]\right)\\
             &=\sum_{\substack{1\leq j \leq n\\m\in \mathbb N ^n}}b_m^{(j)}[Y,z^m\dfrac{\partial}{\partial z_j}].
         \end{align*}
         Therefore, in our case if $X$ is holomorphic on $\mathbb C^n$, we can write
         \begin{align*}
             [z^m\cfrac{\partial}{\partial z_k},X]=\sum_{\substack{1\leq j \leq n\\l\in \mathbb N ^n}}b_l^{(j)}[z^m\cfrac{\partial}{\partial z_k},z^l\dfrac{\partial}{\partial z_j}]
         \end{align*}
         and conclude thanks to the previous point.\\
         $(v)$ : By Hartog's extension theorem (see \cite{broder_theory_nodate} for references), every holomorphic vector field on $W=\mathbb C^n \setminus 0$ extends uniquely to a holomorphic vector field on $\mathbb C^n$. Therefore, the space $H^0(W, \theta)$ of holomorphic vector fields on $W$ is exactly the space of holomorphic vector fields on $\mathbb C^n$. A vector field $X \in H^0(W, \theta)$ can therefore be written 
         \begin{align*}
             X=\sum_{\substack{1\leq j \leq n\\m\in \mathbb N ^n}}b_m^{(j)}z^m\dfrac{\partial}{\partial z_j}&=\sum_{\substack{1\leq j \leq n\\m\in \mathbb N ^n\\\lambda_j=(m,\lambda)}}b_m^{(j)}z^m\dfrac{\partial}{\partial z_j}+\sum_{\substack{1\leq j \leq n\\m\in \mathbb N ^n\\\lambda_j\neq(m,\lambda)}}b_m^{(j)}z^m\dfrac{\partial}{\partial z_j}\\
             &=X^r+X^{nr}
         \end{align*}
         where $$X^r=\sum_{\substack{1\leq j \leq n\\m\in \mathbb N ^n\\\lambda_j=(m,\lambda)}}b_m^{(j)}z^m\dfrac{\partial}{\partial z_j} \in \mathcal{g}_\lambda \quad \text{ and } \quad X^{nr}=\sum_{\substack{1\leq j \leq n\\m\in \mathbb N ^n\\\lambda_j\neq(m,\lambda)}}b_m^{(j)}z^m\dfrac{\partial}{\partial z_j} \in \mathcal{g}_{\lambda}^\perp.$$
         This decomposition is unique by properties of power series (if $X^r+X^{nr}=0$, then $X=\sum_{\substack{1\leq j \leq n\\m\in \mathbb N ^n}}b_m^{(j)}z^m\dfrac{\partial}{\partial z_j}=0$ so for every $j\in \llbracket1,n\rrbracket$ and $m\in \mathbb N^n$, $b_m^{(j)}=0$ thus $X^r=0$ and $X^{nr}=0$).
     \end{proof}
\begin{corollary}
    \label{cor:4.5.1}
    The vector spaces $\text{Vect}_\mathbb C (\xi)$ and $L_\xi(H^0(W, \theta))$ form a direct sum.
\end{corollary}
\begin{proof}
    Suppose there exists $a\in \mathbb C$ and $X$ a holomorphic vector field on $\mathbb C^n$ such that $a\xi=L_\xi(X)$.
    By the previous proposition, we can assume $X$ belongs $\mathcal{g}_\lambda$ to since if we write $X=X^r+X^{nr}$, then $L_\xi(X^r)+L_\xi(X^{nr})=a\xi$ so by points $(iv)$ and $(v)$, $a\xi=L_\xi(X^r)$.
    Write $X=\sum_{\substack{1\leq j \leq n\\m\in \mathbb N ^n\\\lambda_j=(m,\lambda)}}b_{m}^{(j)}z^m\dfrac{\partial}{\partial z_j}$ and $\xi=\xi_0+\xi_R$ where 
    $$\xi_0=\sum_{j=1}^n\lambda_jz_j \tfrac{\partial}{\partial z_j} \quad \text{ and } \quad \xi_R =\sum_{(s,l)\in R} a_{s,l}z^l\dfrac{\partial}{\partial z_s}.$$
    Then the equation $a \xi=L_\xi(X)$ implies :
    \begin{align*}
        a\xi&=[\xi,X]=[\xi_0,X]+[\xi_R, X]=[\xi_R, X]=\sum_{\substack{\lambda_j=(m,\lambda)\\(s,l)\in R}}b_{m}^{(j)}a_{s,l}[z^l\dfrac{\partial}{\partial z_s}, z^m\dfrac{\partial}{\partial z_j}]\\
        &=\sum_{\substack{\lambda_j=(m,\lambda)\\(s,l)\in R\\m_s\neq0}}b_{m}^{(j)}a_{s,l}m_sz^{m+l-e_s}\dfrac{\partial}{\partial z_j}-\sum_{\substack{\lambda_j=(m,\lambda)\\(s,l)\in R\\l_j\neq0}}b_{m}^{(j)}a_{s,l}l_jz^{m+l-e_j}\dfrac{\partial}{\partial z_s}.
    \end{align*}
    We are searching for the monomial vector field $z_1\dfrac{\partial}{\partial z_1}$ in the right-hand side member. 
    \begin{itemize}
        \item In the first sum, it would correspond to $j=1$ and $m+l=e_1+e_s$ so $m=e_1$ and $l=e_s$ or $m=e_s$ and $l=e_1$ (since $m$ and $l$ are both non-zero).\\
        If $m=e_1$, then since $m_s\neq0$, $s=1$ and $(1,e_1)\in R$ which is impossible.\\
        If $m=e_s$ and $l=e_1$, then since $\lambda_1=(m,\lambda)=\lambda_s$ and $(s,l) \in R$, it comes $(1,e_1)\in R$ which is again impossible.
        \item In the second sum, it would correspond to $s=1$ but this is impossible since $R$ does not contain any element $(1,m)$ where $m\in \mathbb N^n$.
    \end{itemize}
    Therefore, since $\lambda_1\neq 0$ (because $0$ does not belong to the convex hull of $\{\lambda_1, \cdots, \lambda_n\}$), it comes $a=0$ and the result.
\end{proof}

An idea of deformation of the transversely holomorphic foliation $\mathcal{F}_0$ coming from the intersection of the holomorphic foliation $\mathcal{F}=\mathcal{F}(\xi)$ with the sphere $S^{2n-1}$ can come from the intersections of holomorphic foliations $\mathcal{F}^S$ induced by a holomorphic family $(\xi_s)_{s\in S}$ of holomorphic non-vanishing vector fields on $W$ with the sphere, such that $\xi_0=\xi$.
We will need the following lemma :
\begin{proposition}
    \label{prop:4.6}
    If $U$ is a neighborhood of $0$ in $\mathcal{g}_\lambda$ taken sufficiently small, then every element $X$ of $U$ is nowhere vanishing on $\mathbb C^n\setminus 0$ and transverse to $S^{2n-1}$ (thus to $S^{2n-1}(r)$ for every $r\in (0,1]$).
\end{proposition}
\begin{proof}
    If $U$ is sufficiently small, then every element $X$ of $U$ satisfies $$\sup_{z \in S^{2n-1}}|\langle X(z),z \rangle|<\dfrac{1}{2}\min_{z \in S^{2n-1}} |\langle \xi(z), z\rangle| \neq 0.$$ Therefore, for every $z\in S^{2n-1}$, $(\xi+X)(z)$ can't be orthogonal to $z$ (for the usual hermitian metric on $\mathbb C^n$), otherwise  :
    $$|\langle X(z), z\rangle| = |\langle \xi(z), z\rangle| <\dfrac{1}{2}\min_{z \in S^{2n-1}} |\langle \xi(z), z\rangle|$$
    which is absurd.\\
    As for the non-vanishing of $\xi+X$ on $\mathbb C^n \setminus 0$, we affirm that :
    \begin{claim}
    \leavevmode
    \begin{enumerate}[label=(\roman*)]
        \item There exists $\epsilon_1>0$ and $\alpha_1>0$ such that for every $X \in \mathcal{g}_\lambda$ whose components are of modulus less than $\epsilon_1$, $(\xi+X)(z)\neq 0$ for every $z \in B(0,\alpha_1)\setminus0$;
        \item There exists $\epsilon_2>0$ and $\alpha_2>\alpha_1$ such that for every $X \in \mathcal{g}_\lambda$ whose components are of modulus less than $\epsilon_2$, $|(\xi+X)(z)|\geq 1$ for every $z \in \mathbb C^n \setminus(B(0,\alpha_2))$.
    \end{enumerate}
    \end{claim}
    \begin{proof}[Proof of the claim]
        $(i)$ : Otherwise, there would exist a sequence $(z_k)_k$ in $\mathbb C^n \setminus 0$ converging to $0$ and an element $X_k\in \mathcal{g}_\lambda$ whose components are less than $|z_k^{(l_k)}|$, where
        $$l_k=\min\{j \in \llbracket1, n \rrbracket, \; z^{(j)}_k\neq 0\},$$
        such that $\xi(z_k)=-X_k(z_k).$  
        After extraction we can assume $(l_k)_k$ is constant equal to $l\in \llbracket1,n\rrbracket$. However, the previous equality implies that $\lambda_lz_k^{(l)}=o(z_k^{(l)})$ which is absurd since $\lambda_l \neq 0$.\\
        $(ii)$ : In the same way, otherwise there would exist a sequence $(z_k)_k$ in $\mathbb C^n \setminus 0$ converging to $+\infty$ in norm and an element $X_k\in \mathcal{g}_\lambda$ whose components are less than $\dfrac{1}{|z_k|^M}$, where $M$ is greater than the largest order of resonance of $(\lambda_1, \ldots,\lambda_n)$, such that $$\|\xi(z_k)+X_k(z_k)\|\leq 1.$$
        However, since $(X_k(z_k))_k$ converges to $0$ as $k$ goes to $+\infty$ (by choice of the components of $X_k$),  $(\xi(z_k))_k$ is bounded but this is absurd as it tends to $+\infty$ in norm.
    \end{proof}
    Now take $U$ smaller so that additionally the components of every of its elements $X$ are less than $\min(\epsilon_1, \epsilon_2)$. By the previous claim, $X$ is such that $\xi+X$ is non zero on $B(0,\alpha_1)$ and on $\mathbb C^n \setminus(B(0,\alpha_2))$.
    We can shrink $U$ even more so that such $X$ satisfies moreover
    $$\sup_{\alpha_1 \leq \|z\|\leq \alpha_2}\|X(z)\|<\dfrac{1}{2}\min_{\alpha_1 \leq \|z\|\leq \alpha_2} \|\xi(z)\| \neq 0.$$
    Then for any $z \in \mathbb C^n$ such that $\alpha_1 \leq \|z\|\leq \alpha_2$, $(\xi +X)(z)$ is non zero (because $\xi(z)$ is not zero).
\end{proof}
\begin{remark}
    As mentioned by Haefliger in \cite{haefliger_deformations_1985}, a small open neighborhood of $0$ in a vector subspace of $\mathcal{g}_\lambda$ complementary to $L_\xi(\mathcal{g}_\lambda)$ parametrizes a versal deformation of the holomorphic vector field $\xi$ (see \cite{szucs_geometrical_1996} and \cite{brushlinskaya_finiteness_1971} where the notion of versal deformation of holomorphic vector field is studied and its existence is proved).
\end{remark}
We are now ready to restate the main theorem. \\
Let $n\geq 2$ and $\mathcal{F}_0=\mathcal{F}_0(\xi)$ a transversely holomorphic foliation on the unit sphere $S^{2n-1} \subset \mathbb C^n$ coming from the intersection with $S^{2n-1}$ of the foliation $\mathcal{F}$ induced by a holomorphic non vanishing vector field defined in a neighborhood of the unit ball $B$ of $\mathbb C^n$. Denote by $\lambda=(\lambda_1, \cdots, \lambda_n)$ the set of eigenvalues of the differential at $0$ of $\xi$ counted with multiplicity. As we mentioned earlier in section \ref{sec:3}, we can assume that :
    \begin{itemize}
    \item $\xi$ is equal to the holomorphic polynomial vector field $$\sum_{j=1}^n\lambda_jz_j \dfrac{\partial}{\partial z_j}+\sum_{(j,m)\in R} a_{j,m}z^m\dfrac{\partial}{\partial z_j};$$
    \item the coefficients $a_{j,m}$ are as small as we want ;
    \item $\xi$ is transverse to  $S^{2n-1}(r)$ for every $r\in (0,1]$;
    \item $\mathcal{F}_0$ is given by the intersection of $\mathcal{F}=\mathcal{F}(\xi)$ with $S^{2n-1}$.
\end{itemize}
Also, we know that the set of eigenvalues $(\lambda_1, \cdots, \lambda_n)$ of the differential at $0$ of $\xi$ belongs to the Poincaré domain.\\
As a result, by what precedes, the remarks at the beginning of Section \ref{sec:3} and the remark after Corollary \ref{cor:2.5.1}, Theorem \ref{thm:defo} can be restated in the following way with the previous notations :
\begin{mythm}{A'}
    \label{thm:defo2}
    Let $S$ be a neighborhood of $0$ in any vector subspace of $\mathcal{g}_\lambda$ complementary to the vector subspace of $\mathcal{g}_\lambda$ generated by $\xi$ and $L_\xi(\mathcal{g}_\lambda)$.\\
    Then, if $S$ is taken sufficiently small, the (well-defined) family $\left(\mathcal{F}_0(\xi+s) \right)_{s \in S}$ of transversely holomorphic foliations on $S^{2n-1}$ coming from the intersections with the sphere $S^{2n-1}$ of the holomorphic foliations $(\mathcal{F}(\xi +s))_{s\in S}$ induced by the family of (non-vanishing on $W=\mathbb C^n\setminus 0$) holomorphic vector fields $(\xi +s)_{s\in S}$ is a versal deformation of the transversely holomorphic foliation $\mathcal{F}_0$ parametrized by $(S,0)$.
\end{mythm}
Indeed, if we prove this result, then by going backwards, we can send the germ of deformation $(\mathcal{F}_0(\xi +s))_{s \in S}$ of $\mathcal{F}_0(\xi)$ to the initial one on $\partial M$ by a smooth family of diffeomorphisms, thanks in particular to Proposition \ref{prop:4.6}.
\section{Proof of the theorem}
\label{sec:5}
Denote by $\mathcal{F}$ the holomorphic foliation on $W:=\mathbb C^n \setminus 0$ induced by the non-vanishing holomorphic flow 
$$\xi=\sum_{j=1}^n\lambda_jz_j \dfrac{\partial}{\partial z_j}+\sum_{(j,m)\in R} a_{j,m}z^m\dfrac{\partial}{\partial z_j};$$
and by $\mathcal{F}_0$ the transversely holomorphic foliation on $S^{2n-1}$ obtained by intersecting the orbits of $\xi$ with the sphere $S^{2n-1}$.\\
As before, consider the following exact sequences of sheaves :
\begin{align}
    0 \longrightarrow \sigma^{tr}_{\mathcal{F}} \longrightarrow &\sigma \xrightarrow{L_{\xi}} \sigma \longrightarrow 0 \label{eq:1}\\
    0 \longrightarrow \theta^{\xi} \longrightarrow &\theta \xrightarrow{L_{\xi}} \theta \longrightarrow 0 \label{eq:2}\\
    0 \longrightarrow \sigma^{tr}_{\xi} \xrightarrow{m_\xi} &\theta^{\xi} \longrightarrow \theta^{tr}_{\mathcal{F}} \longrightarrow 0 \label{eq:3}
\end{align}
\begin{lemma}
    \label{lem:5.1}
    If we note $\iota: S^{2n-1} \to W$ the inclusion, then $\sigma^{tr}_{\mathcal{F}_0}\cong\iota^{-1}\sigma^{tr}_{\mathcal{F}}$ and $\theta^{tr}_{\mathcal{F}_0}\cong\iota^{-1}\theta^{tr}_{\mathcal{F}}$.
\end{lemma}
\begin{proof}
    By definition, $\sigma^{tr}_{\mathcal{F}_0}$ is the unique sheaf of functions with values in $\mathbb C$ whose value for a distinguished chart $(S^{2n-1}\cap U_i, s_i)$ of $\mathcal{F_0}$ is $s_i^{-1}(\sigma_{\mathbb C^{n-1}})$ (see \cite{perrin_algebraic_2008}). 
    On the other hand, $\sigma^{tr}_{\mathcal{F}}$ is a sheaf of functions with values in $\mathbb C$ whose value for a distinguished chart $(U_i, f_i)$ of $\mathcal{F}$ is $f_i^{-1}(\sigma_{\mathbb C^{n-1}})$. Therefore, $\iota^{-1}\sigma^{tr}_{\mathcal{F}}$ is naturally isomorphic to a sheaf of functions with values in $\mathbb C$ whose value for a distinguished chart $(S^{2n-1}\cap U_i, s_i)$ of $\mathcal{F}_0$ is $(f_i\circ \iota)^{-1}(\sigma_{\mathbb C^{n-1}})=s_i^{-1}(\sigma_{\mathbb C^{n-1}})$ by construction of $\mathcal{F}_0$.
    Thus, $\sigma^{tr}_{\mathcal{F}_0}\cong\iota^{-1}\sigma^{tr}_{\mathcal{F}}$.\\
    The exact same proof applies for $\theta^{tr}_{\mathcal{F}_0}$ and $\theta^{tr}_{\mathcal{F}}$ by changing each $\sigma$ into a $\theta$.
\end{proof}
Consider the orbits of the holomorphic flow $(\phi^t)_{t \in \mathbb C}$ generated by $\xi$ (see Proposition \ref{prop:3.3}). Since $0 \in \mathbb C$ does not belong to the convex hull generated by $\{\lambda_1,\cdots, \lambda_n\}$, there exists an angle $\alpha \in \mathbb R$ such that $e^{i\alpha}\lambda_1, \cdots, e^{i\alpha}\lambda_n$ have strictly negative real parts. Therefore, $(\phi^{e^{i\alpha}\cdot t})_{t \in \mathbb R}$ defines a non-vanishing smooth flow on $W$ whose orbits tend to $0$ when $t\in \mathbb R$ goes to $+\infty$ and tend to $+\infty$ in norm when $t$ goes to $-\infty$. These orbits are also tangent to the leaves of $\mathcal{F}$ and transversal to the spheres $S(0,r)$ for every $r \in (0,1]$. Therefore, we can define a smooth projection map $\pi :W \to S^{2n-1}$ which maps the point $z \in W$ to the (unique) point of intersection of the orbit $(\phi^{e^{i\alpha}\cdot t}(z))_{t \in \mathbb R}$ with the sphere $S^{2n-1}$.
\begin{proposition}
    \label{prop:5.2}
    The map $\pi : W \to S^{2n-1}$ is a deformation retract of $W$ onto $S^{2n-1}$. \\
    In particular, it defines a homotopy between $\iota\circ \pi :W\to W$ and $Id_W$ the identity map on $W$.
\end{proposition}
\begin{proof}
    By definition, we have $\pi \circ \iota=Id_{S^{2n-1}}$. For $z \in W$, denote by $t_z \in \mathbb R$ the unique real number $t$ such that $\pi(z)=\phi^{e^{i\alpha}t}(z)$. The map $$\left \{\begin{array}{ccl}
        W & \to & \mathbb R  \\
        z & \mapsto & t_z
    \end{array} \right .$$ is continuous. 
    Let $H:W \times [0,1] \to W$ the map defined for $z\in W$ and $s\in [0,1]$ by :
    $$H(z,s)=\phi^{e^{i\alpha}st_z}(z).$$
    It is not difficult to see that $H$ defines a homotopy between $\iota \circ \pi :W\to W$ and $Id_W$ which satisfies moreover, for every $z\in S^{2n-1}$ and $s\in [0,1]$, $H(z,s)=z$.
\end{proof}
In the following, every vector space and every morphism between vector spaces (in particular our cohomology groups) is understood to be with regards to the field $\mathbb C$ of complex numbers.
\begin{proposition}
    \label{prop:5.3}
    For every $k\in \mathbb N$ :
    \begin{enumerate}[label=(\roman*)]
        \item the map $$\iota^{\# k} : H^k(W,\sigma ^{tr}_{\mathcal{F}})\to  H^k(S^{2n-1},\sigma ^{tr}_{\mathcal{F}_0}),$$
    induced by the inclusion $\iota: S^{2n-1} \to W$, is an isomorphism of vector spaces ;
        \item the map 
    $$\iota^{\Diamond k} : H^k(W,\theta ^{tr}_{\mathcal{F}})\to  H^k(S^{2n-1},\theta ^{tr}_{\mathcal{F}_0}),$$
    induced by the inclusion $\iota: S^{2n-1} \to W$, is an isomorphism of vector spaces.
    \end{enumerate}
\end{proposition}
    \begin{proof}
        The following is a result from algebraic topology which describes how homotopy equivalence behaves with respect to sheaf cohomology groups (see \cite{schapira_algebra_2007}) :
        \begin{theorem}
        \label{thm:5.4}
            Let $f_0,f_1 :X\to Y$ two homotopic maps and let $\mathcal{G}$ a locally constant sheaf on $Y$ of modules over a field $\mathbb K$.\\
            Then for every $k \in \mathbb N$, there exists an isomorphism 
            $$\beta^k: H^k(X,f_0^{-1}\mathcal{G}) \to H^k(X,f_1^{-1}\mathcal{G})$$
            such that $f_1^{\#k}=\beta^k\circ f_0^{\#k}$.
        \end{theorem}
        \begin{corollary}
            Let $f:X\to Y$ a homotopy equivalence and $\mathcal{G}$ a locally constant sheaf on $Y$ of modules over a field $\mathbb K$.\\
            Then for every $k \in \mathbb N$, the map $$f^{\# k} : H^k(X,f^{-1}\mathcal{G})\to  H^k(Y,\mathcal{G})$$
            induced by $f$, is an isomorphism of vector spaces.
        \end{corollary}
    Recall that a locally constant sheaf $\mathcal{G}$ on a topological space $X$ is a sheaf on $X$ satisfying : for every $x\in X$, there exists a neighborhood $U$ of $x$ in $X$ such that the restriction of $\mathcal{G}$ to $U$ is a constant sheaf on $U$. Recall also that a constant sheaf is a sheaf whose stalks are all equal.\\ 
    In our case, by Proposition \ref{prop:5.2}, the inclusion $\iota:S^{2n-1} \to W$ is a homotopy equivalence, but even though $\iota^{-1}\sigma^{tr}_\mathcal{F}=\sigma^{tr}_{\mathcal{F}_0}$ by Lemma \ref{lem:5.1}, $\sigma^{tr}_\mathcal{F}$ is a priori not a locally constant sheaf on $W$. 
    However, upon closely examining the proof of Theorem \ref{thm:5.4}in  \cite{schapira_algebra_2007}, we observe that the assumption that $\mathcal{G}$ is locally constant can be weakened to the following condition : $\mathcal{G}$ is a sheaf of modules over a field $\mathbb K$ such that, if we note $H:X\times [0,1] \to Y$ the homotopy between $f_0:X\to Y$ and $f_1 : X\to Y$, and for $x \in X$ we call $I_x:=\{x\} \times [0,1] \subset X \times [0,1]$, then for every $x\in X$, the restriction of the sheaf $H^{-1}\mathcal{G}$ on $X \times [0,1]$ to $I_x$ is a locally constant sheaf on $I_x$.\\
    That being said we only need to prove the following :
    \begin{claim}
        With the previous notations,  for every $z\in W$, the restriction of the sheaf $H^{-1}\sigma^{tr}_\mathcal{F}$ on $W \times [0,1]$ to $I_z$ is a locally constant sheaf on $I_z$. \\
        The same is true for $\theta^{tr}_\mathcal{F}$ (and even $\theta^\xi$) instead of $\sigma^{tr}_\mathcal{F}$.
    \end{claim}
    \begin{proof}[Proof of the claim]
        Fix $z \in W$. By compactness of $[0,1]$ and continuity of $H$, there exists $N \in \mathbb N$ and $U_0, \ldots, U_N$ distinguished charts for $\xi$ such that $H(z,[0,1]) \subset \bigcup_{i=0}^N U_i$.
        We therefore have an open cover of $I_z$ by $(H\circ i_z)^{-1}(U_0), \ldots, (H\circ i_z)^{-1}(U_N)$, where $$i_z : \{z\} \times [0,1]\to W \times [0,1]$$ is the inclusion.
        Since $\restr{\sigma^{tr}_\mathcal{F}}{U_i}=f_i^{-1}(\sigma_{\mathbb C^{n-1}})$, it comes that the restriction of the sheaf $H^{-1}\sigma^{tr}_\mathcal{F}$ to $(H\circ i_z)^{-1}(U_i)$ is
        $$\restr{i_z^{-1}\left(H^{-1}\sigma^{tr}_\mathcal{F} \right)}{(H\circ i_z)^{-1}(U_i)}=(f_i \circ H\circ i_z)^{-1}(\sigma_{\mathbb C^{n-1}})$$
        which is a constant sheaf on $(H\circ i_z)^{-1}(U_i)$ since its stalk at a point $t$ is the stalk  of
        $\sigma_{\mathbb C^{n-1}}$
        at $f_i(H(z,t))$, which does not depend on $t \in (H\circ i_z)^{-1}(U_i)$ by definition of $H$ (for every $t\in [0,1]$, $H(z,t)$ stays in the orbit of $z$).
        The same is obviously true for $\theta^{tr}_\mathcal{F}$ and $\theta^\xi$.
    \end{proof}
    \end{proof}
\begin{corollary}
    Denote by $\rho_{\mathcal{F}} : T_0S 
    \to H^1(M, \theta^{tr}_{\mathcal{F}})$ the Kodaira-Spencer map measuring the germ of deformation of the transversely holomorphic foliation of $\mathcal{F}$ induced by the holomorphic family of vector fields $(\xi +s)_{s \in S}$.
    Denote also by $\rho_{\mathcal{F}_0} : T_0S 
    \to H^1(M, \theta^{tr}_{\mathcal{F}_0})$ the Kodaira-Spencer map measuring the germ of deformation of the transversely holomorphic foliation $\mathcal{F}_0$ induced by the intersection with the sphere of the holomorphic family of vector fields $(\xi +s)_{s \in S}$.\\
    Then $\rho_{\mathcal{F}_0}=\iota^{\Diamond1} \circ \rho_\mathcal{F}$.
    \end{corollary}
\begin{proof}
    This is straightforward by definition.
\end{proof}
In order to prove Theorem \ref{thm:defo2}, we will use Corollary \ref{cor:2.5.1} of \cite{haefliger_deformations_1983} and prove that $\rho_{\mathcal{F}_0}: T_0S \to H^1(M, \theta^{tr}_{\mathcal{F}_0})$ 
is an isomorphism, since we already know that $(S,0)$ is non-singular because $S$ is an open subset of a vector space, thus smooth.
By the previous corollary, it is the same as proving that $\rho_{\mathcal{F}} : T_0S 
\to H^1(M, \theta^{tr}_{\mathcal{F}})$ is an isomorphism.
\begin{proposition}
    The Kodaira-Spencer map $\rho_\mathcal{F}:T_0S \to H^1(W, \theta^{tr}_\mathcal{F})$ measuring the germ of deformation of the transversely holomorphic foliation of $\mathcal{F}$ induced by the holomorphic family of vector fields $(\xi+s)_{s\in S}$ is given, for $X\in T_0S\subset \mathcal{g_\lambda}\subset H^0(W, \theta)$, by :
$$\rho_\mathcal{F}(X)=-p\circ \delta'(X).$$
\end{proposition}
\begin{proof}
Consider a holomorphic family of non-vanishing holomorphic fields $(\xi +s)_{s \in S}$ on $W$. 
Denote by $d$ the (complex) dimension of $T_0S$ and by $(X_1, \cdots, X_d)$ a basis of holomorphic vector fields of $T_0S \subset \mathcal{g_\lambda}$. Since $S$ is an open neighborhood of $0$ in a vector subspace of $\mathcal{g}_\lambda$, then it can be seen as an open neighborhood of $0\in T_0S$ in $T_0S$.
Therefore, if $\dfrac{\partial}{\partial s}=\sum_{l=1}^d c_lX_l \in T_0S$ is an element of $T_0S$, where $(c_1, \cdots, c_d) \in \mathbb C^d$, and if we write $s=\sum_{l=1}^d b_lX_l \in S$, where $(b_1, \cdots, b_d)\in \mathbb C^d$ is small enough, then by definition :
$$\restr{\dfrac{\partial (\xi+s)}{\partial s}}{s=0}=\sum_{l=1}^d c_lX_l \in H^0(W, \theta).$$
By Proposition \ref{prop:2.9}, we conclude.
\end{proof}
Consequently, we only need to prove that the restriction to $T_0S$ of the composition $p\circ \delta': H^0(W, \theta) \to H^1(W, \theta^{tr}_\mathcal{F})$ is an isomorphism from $T_0S$ to $H^1(W, \theta^{tr}_\mathcal{F})$.\\
Reexamine the cohomology long exact sequences associated to the exact sequences of sheaves (see Proposition \ref{prop:2.8}), which are represented in the following commutative diagram: 
   \[\begin{tikzcd}[column sep=5.87mm]
        & & & \vdots \arrow[d] & & & \\
        \cdots \arrow[r] & H^0(W, \sigma) \arrow[r,"L_\xi"] \arrow[d] & H^0(W,\sigma) \arrow[r, "\delta"] \arrow[d] & H^1(W,\sigma^{tr}_\mathcal{F}) \arrow[r] \arrow[d,"m_\xi"] & H^1(W,\sigma) \arrow[r,"L_\xi"] \arrow[d] & H^1(W,\sigma) \arrow[r] \arrow[d]  &\cdots \\
        \cdots \arrow[r]
 &  H^0(W, \theta)  \arrow[r,"L_\xi"] &  H^0(W,\theta) \arrow[r, "\delta'"] \arrow[rd, swap, "p\circ\delta' "] & H^1(W,\theta^\xi) \arrow[r] \arrow[d, "p"] & H^1(W,\theta) \arrow[r,"L_\xi"] & H^1(W,\theta) \arrow[r] & \cdots\\
  & &  &  H^1(W, \theta^{tr}_\mathcal{F}) \arrow[d]  &  &  & \\
  & &  &  H^2(W, \sigma^{tr}_\mathcal{F}) \arrow[d]  &  &  & \\
  & &  &  \vdots &  &  & 
    \end{tikzcd} 
    \]

\begin{proposition}
    \leavevmode
    \begin{enumerate}[label=(\roman*)]
        \item For every $k$ different from $0$ and $n-1$, $H^k(W,\sigma)=H^k(W,\theta)=0$ ;
        \item $H^{n-1}(W,\sigma)$ (respectively $H^{n-1}(W,\theta)$) is isomorphic to the vector space of convergent series  $\sum\limits_{\forall i, m_i<0} a_m z^m$ (respectively $\sum\limits_{\forall i, m_i<0} a_m^s z^m \dfrac{\partial}{\partial z_s}$)  on $(\mathbb C^*)^n$ ;
        \item For every $k>0$, the maps $$L_{\xi} : H^{n-1}(W,\sigma) \to H^{n-1}(W,\sigma)  \quad \text{ and } \quad L_{\xi} : H^{n-1}(W,\theta) \to H^{n-1}(W,\theta)$$ are isomorphisms of vector spaces ;
        \item The kernel of the map $L_\xi : H^0(W, \sigma) \to H^0(W, \sigma)$ is $\text{Vect}_\mathbb C(1)$, and its image is the space of holomorphic functions on $\mathbb C^n$ vanishing at $0$ ;
        \item The map $\restr{L_\xi}{\mathcal{g}_\xi^\perp} : \mathcal{g}_\xi^\perp \to \mathcal{g}_\xi^\perp$ is an isomorphism.
    \end{enumerate}
    
\end{proposition}
\begin{proof}
    The points $(i)$ and $(ii)$ are proven in \cite{haefliger_deformations_1985}. We detail the proof for completeness.\\
    For $j \in \llbracket1,n \rrbracket$, let $U_j=\{(z_1, \cdots, z_n) \in W, \; z_j\neq0\} \cong \mathbb C^{n-1}\times \mathbb C^*$. Since $\mathbb C$ and $\mathbb C^*$ are both Stein manifolds, the product $U_j$ is also a Stein manifold (see \cite{grauert_theory_2004} for a reminder on Stein spaces and cohomology theory). Therefore, any finite intersection between the open sets $U_j$ for $j \in \llbracket1,n \rrbracket$ is a Stein manifold.
    Thus, $\mathcal{U}=(U_j)_{1\leq j \leq n}$ is an open cover of $W$ by Stein open sets. Since $\sigma$ is a coherent sheaf on $W$, by the theorem of Leray the cohomology groups $H^k(W, \sigma)$ can be computed using alternate cochains (the same is true for any intersection between the $U_j$'s). Since $\mathcal{U}$ is composed of $n$ elements, $H^k(W, \sigma)=0$ for every $k\geq n$.\\
    $(i)$ : We prove by induction on $n \geq 3$ that $H^k(\mathbb C^n \setminus0, \sigma)=0$ for every $k$ such that $0<k<n-1$.
    The idea is to write $$\mathbb C^n \setminus 0 =(\mathbb C^{n-1}\setminus 0 \times \mathbb C )\cup (\mathbb C^{n-1} \times \mathbb C^*),$$ use the Mayer-Vietoris sequence associated to this cover (see \cite{kashiwara_sheaves_1994}) as well as the Künneth formula (see \cite{sampson_kunneth_1959}), the fact that the cohomology groups (for $j \geq 1$) of each Stein manifold with respect to the coherent sheaf of holomorphic functions is zero and the induction hypothesis.\\
    Recall the Mayer-Vietoris sequence, and the Künneth formula :
    \begin{theorem}[Mayer-Vietoris]
        Let $\mathcal{G}$ a sheaf of functions on a topological space $X$. Let $U, V$ two open sets of $X$ whose union covers $X$. Denote by $\alpha$ the map which sends $s\in \mathcal{G}(\Omega)$ to the couple $(\restr{s}{\Omega \cap U}, \restr{s}{\Omega \cap V}) \in \mathcal{G}(\Omega \cap U) \times \mathcal{G}(\Omega\cap V)$, where $\Omega$ is an open subset of $X$ ; and by $\beta$ the map which sends $(s,s')\in\mathcal{G}(\Omega \cap U) \times \mathcal{G}(\Omega\cap V)$ to $\restr{s'}{\Omega \cap U\cap V}-\restr{s}{\Omega \cap U\cap V} \in \mathcal{G}(\Omega \cap U \cap V)$.\\
        Then, if we note $\alpha_k$ and $\beta_k$ the induced map between the $k$-th cohomology groups, there is an exact sequence 
        \[
\begin{tikzcd}
0 \ar[r] & H^0(X, \mathcal{G}) \ar[r,"\alpha_0"] & H^0(U,\restr{\mathcal{G}}{U}) \oplus H^0(V,\restr{\mathcal{G}}{V}) \ar[r,"\beta_0"] & H^0(U\cap V,\restr{\mathcal{G}}{U \cap V})\ar[dll,swap,"\delta_0"] \\
& H^1(X, \mathcal{G}) \ar[r,"\alpha_1"] & H^1(U,\restr{\mathcal{G}}{U}) \oplus H^0(V,\restr{\mathcal{G}}{V}) \ar[r,"\beta_1"] & H^1(U\cap V,\restr{\mathcal{G}}{U \cap V})\ar[r] & \cdots
\end{tikzcd}
\]
    \end{theorem}
    \begin{theorem}[Künneth]
        Let $\mathcal{F}$ (respectively $ \mathcal{G}$) a coherent analytic sheaf on a complex manifold $X$ (respectively $Y$). Denote by $\text{pr}_1:X \times Y \to X$ the projection onto the first factor and $\text{pr}_2:X \times Y \to Y$ the projection onto the second factor.\\
        Then for every $k \in \mathbb N$, 
        $$H^k(X \times Y, \text{pr}_1^*\mathcal{F} \otimes_{\mathcal{O}_{X \times Y}}\text{pr}_2^*\mathcal{G}) 
    \cong \bigoplus_{j=0}^k H^j(X,\mathcal{F})\otimes_\mathbb C H^{k-j}(Y, \mathcal{G}),$$
    where the notation $f^*$ refers to the analytic inverse under $f$.
    \end{theorem}
    \noindent In order to simplify the diagram, we will note $\sigma$ the sheaf of holomorphic functions for every open set involved, and write the sheaves $\sigma$ as indexes. In our case, it is immediate to see that for every $X, Y$ involved,
    $$\text{pr}_1^*\mathcal{F} \otimes_{\mathcal{O}_{X \times Y}}\text{pr}_2^*\mathcal{G}=\mathcal{O}_{X \times Y}.$$
    Therefore, there won't be any confusion.\\
    We first prove the result for $n=3$ in order to make things clearer, and then for any $n \geq 3$ by induction. \\
    We write the beginning of the Mayer-Vietoris sequence associated to the cover $$\mathbb C^3 \setminus 0=(\mathbb C^{2}\setminus 0 \times \mathbb C )\cup (\mathbb C^{2} \times \mathbb C^*) :$$
\[
\begin{tikzcd}
0 \ar[r] & H^0_\sigma(\mathbb C^3) \ar[r,"\alpha_0"] & H^0_\sigma(\mathbb C^2\setminus0 \times \mathbb C) \oplus H^0_\sigma(\mathbb C^2 \times \mathbb C^*) \ar[r,"\beta_0"] & H^0_\sigma(\mathbb C^2\setminus0 \times \mathbb C^*) \ar[dll,swap,"\delta_0"] \\
& H^1_\sigma(\mathbb C^3\setminus0) \ar[r,"\alpha_1"] & H^1_\sigma(\mathbb C^2\setminus0 \times \mathbb C) \oplus H^1_\sigma(\mathbb C^2 \times \mathbb C^*)\ar[r,"\beta_1"] & H^1_\sigma(\mathbb C^2\setminus0 \times \mathbb C^*) \ar[dll,swap,"\delta_1"] \\
& H^2_\sigma(\mathbb C^3\setminus0) \ar[r,"\alpha_2"] & H^2_\sigma(\mathbb C^2\setminus0 \times \mathbb C) \oplus H^2_\sigma(\mathbb C^2 \times \mathbb C^*) \ar[r,"\beta_2"] & H^2_\sigma(\mathbb C^2\setminus0 \times \mathbb C^*) \ar[r] & \cdots
\end{tikzcd}
\]
Now, we remark that :
\begin{itemize}
    \item for every $k\geq 1$, $H^k_\sigma(\mathbb C^2 \times \mathbb C^*) = 0$ since $\mathbb C^2 \times \mathbb C^*$ is a Stein manifold by product, and $\sigma$ is a coherent sheaf ;
    \item by the Künneth formula, for every $k\geq 1$,  $$H^k_\sigma(\mathbb C^2 \setminus 0\times \mathbb C)=H^k_\sigma(\mathbb C^2\setminus0) \otimes H^0_\sigma(\mathbb C) \quad \text{ and }\quad H^k_\sigma(\mathbb C^2 \setminus 0\times \mathbb C^*)=H^k_\sigma(\mathbb C^2\setminus0) \otimes H^0_\sigma(\mathbb C^*)$$ since again $\mathbb C$ and $ \mathbb C^*$ are both Stein.
\end{itemize}
Therefore, since $H^k_\sigma(\mathbb C^n \setminus 0)=0$ for every $k \geq n$, the long exact sequence simplifies to :
\[
\begin{tikzcd}
0 \ar[r] & H^0_\sigma(\mathbb C^3) \ar[r,"\alpha_0"] & H^0_\sigma(\mathbb C^2\setminus0 \times \mathbb C) \oplus H^0_\sigma(\mathbb C^2 \times \mathbb C^*) \ar[r,"\beta_0"] & H^0_\sigma(\mathbb C^2\setminus0 \times \mathbb C^*) \ar[dll,swap,"\delta_0"] \\
& H^1_\sigma(\mathbb C^3\setminus0) \ar[r,"\alpha_1"] & H^1_\sigma(\mathbb C^2\setminus0 \times \mathbb C) \oplus 0\ar[r,"\beta_1"] & H^1_\sigma(\mathbb C^2\setminus0 \times \mathbb C^*) \ar[dll,swap,"\delta_1"] \\
& H^2_\sigma(\mathbb C^3\setminus0) \ar[r,"\alpha_2"] & 0\ar[r,"\beta_2"] & 0
\end{tikzcd}
\]
\begin{claim}
    $\beta_0$ is surjective and $\beta_1$ is injective.
\end{claim}
\begin{proof}
    This will be proven more generally in the induction process.
\end{proof}
\noindent By exactness of this long sequence, it comes that $\delta_0$ is the zero map and is also surjective, which implies that $H^1_\sigma(\mathbb C^3\setminus 0)=0$.\\
Now assume, for $n\geq 3$ fixed, that $H^k(\mathbb C^n\setminus 0, \sigma)=0$ for every $k$ such that $0<k<n-1$.\\
We write the Mayer-Vietoris sequence associated to the cover $$\mathbb C^{n+1} \setminus 0=(\mathbb C^{n}\setminus 0 \times \mathbb C )\cup (\mathbb C^{n} \times \mathbb C^*) :$$
\[
\begin{tikzcd}
0 \ar[r] & H^0_\sigma(\mathbb C^{n+1}) \ar[r,"\alpha_0"] & H^0_\sigma(\mathbb C^n\setminus0 \times \mathbb C) \oplus H^0_\sigma(\mathbb C^n \times \mathbb C^*) \ar[r,"\beta_0"] & H^0_\sigma(\mathbb C^n\setminus0 \times \mathbb C^*) \ar[dll,swap,"\delta_0"] \\
& H^1_\sigma(\mathbb C^{n+1}\setminus0) \ar[r,"\alpha_1"] & H^1_\sigma(\mathbb C^n\setminus0 \times \mathbb C) \oplus H^1_\sigma(\mathbb C^n \times \mathbb C^*)\ar[r,"\beta_1"] & H^1_\sigma(\mathbb C^n\setminus0 \times \mathbb C^*) \ar[dll,swap,"\delta_1"] \\
& H^2_\sigma(\mathbb C^{n+1}\setminus0) \ar[r,"\alpha_2"] & H^2_\sigma(\mathbb C^n\setminus0 \times \mathbb C) \oplus H^2_\sigma(\mathbb C^n \times \mathbb C^*) \ar[r,"\beta_2"] & H^2_\sigma(\mathbb C^n\setminus0 \times \mathbb C^*) \ar[dl]\\
& & \cdots \ar[dl] & \\
& H^n_\sigma(\mathbb C^{n+1}\setminus0) \ar[r,"\alpha_n"] & H^n_\sigma(\mathbb C^n\setminus0 \times \mathbb C) \oplus H^n_\sigma(\mathbb C^n \times \mathbb C^*) \ar[r,"\beta_n"] & H^n_\sigma(\mathbb C^n\setminus0 \times \mathbb C^*)
\end{tikzcd}
\]
Now, we remark that :
\begin{itemize}
    \item for every $k\geq 1$, $H^k_\sigma(\mathbb C^n \times \mathbb C^*) = 0$ since $\mathbb C^n \times \mathbb C^*$ is a Stein manifold by product, and $\sigma$ is a coherent sheaf ;
    \item by the Künneth formula, for every $k\geq 1$,  $$H^k_\sigma(\mathbb C^n \setminus 0\times \mathbb C)=H^k_\sigma(\mathbb C^n\setminus0) \otimes H^0_\sigma(\mathbb C) \quad \text{ and }\quad H^k_\sigma(\mathbb C^n \setminus 0\times \mathbb C^*)=H^k_\sigma(\mathbb C^n\setminus0) \otimes H^0_\sigma(\mathbb C^*)$$ since again $\mathbb C$ and $ \mathbb C^*$ are both Stein.
\end{itemize}
Therefore, since $H^k_\sigma(\mathbb C^{n+1} \setminus 0)=0$ for every $k \geq n+1$, the long exact sequence simplifies to :
\[
\begin{tikzcd}
0 \ar[r] & H^0_\sigma(\mathbb C^{n+1}) \ar[r,"\alpha_0"] & H^0_\sigma(\mathbb C^n\setminus0 \times \mathbb C) \oplus H^0_\sigma(\mathbb C^n \times \mathbb C^*) \ar[r,"\beta_0"] & H^0_\sigma(\mathbb C^n\setminus0 \times \mathbb C^*) \ar[dll,swap,"\delta_0"] \\
& H^1_\sigma(\mathbb C^{n+1}\setminus0) \ar[r,"\alpha_1"] & 0 \ar[r,"\beta_1"]& 0 \ar[dll,swap,"\delta_1"] \\
& H^2_\sigma(\mathbb C^{n+1}\setminus0) \ar[r,"\alpha_2"] & 0 \ar[r,"\beta_2"] & 0 \ar[dl] \\
&  & \cdots\ar[dl] & \\
& H^{n-2}_\sigma(\mathbb C^{n+1}\setminus0) \ar[r,"\alpha_{n-2}"] & 0 \ar[r,"\beta_{n-2}"]& 0 \ar[dll,swap,"\delta_{n-2}"] \\
& H^{n-1}_\sigma(\mathbb C^{n+1}) \ar[r,"\alpha_{n-1}"] & H^{n-1}_\sigma(\mathbb C^n\setminus0 \times \mathbb C) \ar[r,"\beta_{n-1}"] & H^{n-1}_\sigma(\mathbb C^n\setminus0 \times \mathbb C^*) \ar[dll,swap,"\delta_{n-1}"] \\
& H^n_\sigma(\mathbb C^{n+1}\setminus0) \ar[r,"\alpha_n"] & 0 \ar[r,"\beta_n"] & 0 
\end{tikzcd}
\]
This proves that $H^k(\mathbb C^{n+1}\setminus0, \sigma)=0$ for every $k$ such that $1<k<n-1$.
\begin{claim}
    $\beta_0$ is surjective and $\beta_{n-1}$ is injective.
\end{claim}
\begin{proof}
    $\beta_0$ is surjective by Hartog's extension theorem since $n\geq 2$. As for the map $\beta_{n-1} : H^{n-1}_\sigma(\mathbb C^n\setminus0 \times \mathbb C) \to H^{n-1}_\sigma(\mathbb C^n\setminus0 \times \mathbb C^*) $, recall that it sends the cohomology class of $g_{1,\ldots,n} \in \sigma(U_1 \cap \cdots \cap U_n)$ to the cohomology class of $\restr{g_{1,\ldots,n}}{U_{n+1}} \in \sigma(U_1 \cap \cdots \cap U_n\cap U_{n+1})$, where $U_j=\{(z_1, \cdots, z_n) \in \mathbb C^{n+1}\setminus 0, \; z_j\neq0\}$ for $j \in \llbracket1,n+1 \rrbracket$ (since we can again compute the cohomology groups with the Cëch cohomology).
    Assume $$\restr{g_{1,\ldots,n}}{U_{n+1}}=\sum_{j=1}^{n}(-1)^{j-1}l_{1,\cdots,\hat{j}, \cdots,n}$$
    where, for $j\in \llbracket1,n \rrbracket$, $l_{1,\cdots,\hat{j}, \cdots,n} $ is holomorphic on $U_1\cap \cdots \cap U_{\hat{j}}\cap \cdots \cap U_n \cap U_{n+1}=(\mathbb C^*)^{j-1}\times \mathbb C \times (\mathbb C^*)^{n-j} \times \mathbb C^*$.
    Write $g_{1, \cdots, n}=\sum_{m \in \mathbb Z^{n+1}}a_mz^m$ where $a_m=0$ if $m_{n+1}<0$, and $l_{1,\cdots,\hat{j}, \cdots,n}=\sum_{m\in \mathbb Z^{n+1}}c_m^{(j)}z^m$ where $c_m^{(j)}=0$ if $m_j<0$. Then, by properties of power series, the previous equality gives for every $m \in \mathbb Z^{n+1}$ such that $m_{n+1}<0$ :
    $$0=\sum_{j=1}^{n}(-1)^{j-1}c_m^{(j)}.$$
    Thus
    \begin{align*}
    \restr{g_{1,\ldots,n}}{U_{n+1}}&=\sum_{j=1}^{n}(-1)^{j-1}l_{1,\cdots,\hat{j}, \cdots,n}=\sum_{\substack{m \in \mathbb Z^{n+1}\\m_{n+1}\geq 0}}\left(\sum_{j=1}^{n}(-1)^{j-1}c_m^{(j)}\right)z^m\\
    &=\restr{\sum_{j=1}^{n}(-1)^{j-1}L_{1,\cdots,\hat{j}, \cdots,n}}{U_{n+1}}
    \end{align*}
    where $L_{1,\cdots,\hat{j}, \cdots,n} =\sum_{\substack{m\in \mathbb Z^{n+1}\\m_j\geq 0\\m_{n+1}\geq 0}}c_m^{(j)}z^m $ is holomorphic on $U_1\cap \cdots \cap U_{\hat{j}}\cap \cdots \cap U_n$. By properties of power series, since $U_1\cap \cdots\cap U_n$ is connected, $$g_{1, \ldots, n}=\sum_{j=1}^{n}(-1)^{j-1}L_{1,\cdots,\hat{j}, \cdots,n}.$$
    This proves exactly that the cohomology class of $g_{1, \cdots, n}$ is $0\in  H^{n-1}_\sigma(\mathbb C^n\setminus0 \times \mathbb C)$ and that $\beta_{n-1}$ is therefore injective.
\end{proof}
By exactness of this sequence, it comes that $\delta_0$ is the zero map and is also surjective, which implies that $H^1_\sigma(\mathbb C^{n+1}\setminus 0)=0$ ; and that $\delta_{n-2}$ is the zero map and is also surjective, which implies that $H^{n-1}_\sigma(\mathbb C^{n+1}\setminus 0)=0$.\\
Since $\theta\cong\sigma^{\oplus n}$ and $W$ is covered by the open stein sets $(U_i)_{1\leq i \leq n}$, we can compute the cohomology groups of both these sheaves by Cëch cohomology and prove naturally that for every $k \in \mathbb N$, $$(H^k(W, \sigma))^{\oplus n} \cong H^k(W, \theta).$$
This proves $(i)$.\\
$(ii)$ : An element of $H^{n-1}(W, \sigma)$ is an equivalence class of a holomorphic function $g_{1\cdots n}$ on $U_1 \cap \cdots \cap U_n=(\mathbb C^*)^n$, where two holomorphic functions $g_{1\cdots n}$ and $h_{1,\cdots,n}$ on $(\mathbb C^*)^n$ are equivalent if and only if for every $j\in \llbracket1, n \rrbracket$, there exists a holomorphic function $l_{1,\cdots,\hat{j}, \cdots,n}$ on $(\mathbb C^*)^{j-1}\times \mathbb C \times (\mathbb C^*)^{n-j}$ such that 
$$g_{1, \cdots,n}-h_{1,\cdots,n}=\sum_{j=1}^n(-1)^{j-1}l_{1,\cdots,\hat{j}, \cdots,n}.$$
If we write as power series $g_{1, \cdots, n}=\sum_{m \in \mathbb Z^n}a_mz^m$ and $h_{1, \cdots, n}=\sum_{m \in \mathbb Z^n}b_mz^m$, then the previous equality implies : for every $m\in \mathbb Z$ whose components are all strictly negative, $a_m=b_m$. We can therefore define a map which assigns to the equivalence class of the cocycle $g_{1, \ldots,n}=\sum_{m \in \mathbb Z^n}a_mz^m$ the convergent serie $\sum_{\substack{m\in \mathbb Z^n\\ \forall j,m_j<0}}a_mz^m$ on $(\mathbb C^*)^n$. It is surjective, and if a holomorphic function $g_{1, \ldots,n}=\sum_{m \in \mathbb Z^n}a_mz^m$ satisfies $$\sum_{\substack{m\in \mathbb Z^n\\ \forall j,m_j<0}}a_mz^m=0,$$ then
\begin{align*}
    g=\sum_{\substack{m\in \mathbb Z^n\\ \exists j,m_j\geq0}}a_mz^m=\sum_{j=1}^n\sum_{\substack{m\in \mathbb Z^n\\m_j\geq0}}a_mz^m=\sum_{j=1}^n(-1)^{j-1}l_{1,\cdots,\hat{j}, \cdots,n}
\end{align*}
where, for $j \in \llbracket 1,n \rrbracket$, 
$$l_{1,\cdots,\hat{j}, \cdots,n}=(-1)^{j-1}\sum_{\substack{m\in \mathbb Z^n\\m_j\geq0}}a_mz^m$$
is  holomorphic on $(\mathbb C^*)^{j-1}\times \mathbb C \times (\mathbb C^*)^{n-j}$. It is thus injective also.\\
The remark at the end of the previous point concludes. This proves $(ii)$.\\
$(iii)$ : For every $k\neq 0$ different from $n-1$, the result is immediate by point $(i)$.
First assume $\xi$ is equal to the diagonal vector field 
$$\xi_0=\sum_{j=1}^n\lambda_jz_j \dfrac{\partial}{\partial z_j}.$$
We compute, for $f=\sum_{\substack{m \in \mathbb Z^n\\ \forall i, m_i<0}}a_m z^m \in H^{n-1}(W,\sigma)$ : 
\begin{align*}
    L_\xi(f)=\sum_{j=1}^n\lambda_jz_j \dfrac{\partial f}{\partial z_j}=  \sum_{j=1}^n\sum_{\substack{m \in \mathbb Z^n\\ \forall i, m_i<0}}a_m\lambda_jz_j m_jz^{m-e_j}=\sum_{\substack{m \in \mathbb Z^n\\ \forall i, m_i<0}}a_m(m,\lambda)z^m.
\end{align*}
Now, if we note $$ \mathcal{H}(\lambda):= \left\{\sum_{j=1}^n t_j \lambda_j, \; (t_1, \ldots, t_n)\in (\mathbb R^+)^n, \; \sum_{j=1}^n t_j=1 \right \}$$
the convex hull of $\{\lambda_1, \cdots, \lambda_n\}$, 
the distance $\delta$ (for the modulus) from $0$ to $\mathcal{H}(\lambda)$ is strictly positive since $0\in  \mathbb C$ does not belong to $\mathcal{H}(\lambda)$ and the latter is closed in $\mathbb C$. Therefore, for every $(t_1, \ldots, t_n)\in (\mathbb R^+)^n\cup (\mathbb R^-)^n $,
$$\left | \sum_{j=1}^n t_j \lambda_j\right | \geq \delta\left | \sum_{j=1}^n t_j\right |.$$
That being said, $L_\xi :H^{n-1}(W,\sigma)\to H^{n-1}(W,\sigma)$ is injective since for every $m\in \mathbb Z^n$ whose coefficients are all strictly negative, $(m, \lambda)$ is not zero by the above.\\
As for surjectivity, consider an element $g=\sum_{\substack{m \in \mathbb Z^n\\ \forall i, m_i<0}}b_m z^m \in H^{n-1}(W,\sigma)$. 
For a tuple $m \in \mathbb Z^n$ whose coefficients are all strictly negative, let $a_m=\dfrac{b_m}{(m,\lambda)}$. Then by what precedes, 
$$|a_m|\leq \dfrac{1}{\delta}\cdot \dfrac{|b_m|}{|m|}$$
which shows that $\sum_{\substack{m \in \mathbb Z^n\\ \forall i, m_i<0}}a_m z^m $ defines an element $f$ of $H^{n-1}(W,\sigma)$, satisfying moreover $L_\xi(f)=g$.\\
In the more general case where $\xi$ can be written as
$$\sum_{j=1}^n\lambda_jz_j \dfrac{\partial}{\partial z_j}+\sum_{(j,m)\in R} a_{j,m}z^m\dfrac{\partial}{\partial z_j},$$
to prove that $L_\xi:H^{k}(W,\sigma)\to H^{k}(W,\sigma)$ is an isomorphism for every $k \geq 1$, we use the preceding result for $\xi_0$ as well as the long exact sequences of cohomology groups and an argument coming from the theory of elliptic complexes.\\
First note that, by the long exact sequence of cohomology group, this statement is equivalent to : $$\forall k\geq 2, \; H^k(W, \sigma^{tr}_\mathcal{F})=0 \quad \text{ and } \quad L_\xi:H^{1}(W,\sigma)\to H^{1}(W,\sigma) \text{ is injective}.$$
Since for every $k\geq 1$ the map $L_{\xi_0}:H^{k}(W,\sigma)\to H^{k}(W,\sigma)$ is an isomorphism, it comes, by Proposition \ref{prop:5.3}, that for $k\geq 2$ : $$H^k(W, \sigma^{tr}_{\mathcal{F}(\xi_0)})=H^k(S^{2n-1}, \sigma^{tr}_{\mathcal{F}_0(\xi_0)})=0.$$
We use now the fact, as proved in \cite{duchamp_deformation_1979}, that the cohomology groups $H^k(S^{2n-1}, \sigma^{tr}_{\mathcal{F}_0(\xi)})$ are isomorphic to the cohomology groups of an elliptic complex, so each of them is the kernel of an elliptic differential operator (see \cite{wells_differential_2008}). 
Its (finite) dimension is an upper semi-continuous function of the components of $\xi_R=\xi-\xi_0=\sum_{(j,m)\in R} a_{j,m}z^m\dfrac{\partial}{\partial z_j}$, which as we said can be initially taken as small as we want.
That being said and done, it comes that $H^k(W, \sigma^{tr}_{\mathcal{F}(\xi)})=H^k(S^{2n-1}, \sigma^{tr}_{\mathcal{F}_0(\xi)})=0$ for every $k\geq 2$. We now prove that $L_{\xi}:H^{1}(W,\sigma)\to H^{1}(W,\sigma)$ is injective. If $n\geq 3$, this is immediate by point $(i)$.
In the case $n=2$, we write for $f=\sum_{\substack{m \in \mathbb Z^n\\ \forall i, m_i<0}}b_m z^m \in H^{1}(W,\sigma)$ :
\begin{align*}
    L_\xi(f)=\sum_{\substack{m \in \mathbb Z^2\\ \forall i, m_i<0}}b_m(m,\lambda)z^m+\sum_{\substack{m \in \mathbb Z^2\\ \forall i, m_i<0}}b_mm_2\sum_{(2,l)\in R}a_{2,l}z^{m+l_1e_1-e_2}.
\end{align*}
The second sum does not contain any monomial $z_1^{k_1}z_2^{-1}$ where $k_1<0$. Therefore, if $L_\xi(f)=0$, necessarily for every $m\in \mathbb Z^2$ such that $m_1<0$ and $m_2=-1$, $b_m(m,\lambda)=0$ i.e. $b_m=0$.
By induction on $m_2 \in \mathbb Z_{<0}$, we prove that $b_m=0$ for every $m\in \mathbb Z^2$ whose components are strictly negative.
It suffices to look at the monomial $z_1^{k_1}z_2^{k_2-1}$, for every $k_1<0$, and for $k_2<0$ fixed, knowing that $b_m=0$ for every $m\in \mathbb Z^2$ such that $m_1<0$ and $ k_2\leq m_2\leq-1$.
This monomial corresponds to $m_2=k_2$ in the second sum but it does not appear for any $k_1<0$ since $b_{k_1,k_2}=0$ by induction hypothesis.
Therefore, the first sum gives $b_{k_1,k_2-1}=0$ for any $k_1<0$ which proves the result for the sheaf of holomorphic functions $\sigma$.\\
As for the sheaf of holomorphic vector fields $\theta$, we first assume $\xi$ is equal to the diagonal vector field 
$$\xi_0=\sum_{j=1}^n\lambda_jz_j \dfrac{\partial}{\partial z_j},$$
and  compute, for $X=\sum_{s=1}^n\sum_{\substack{m \in \mathbb Z^n\\ \forall i, m_i<0}}b_m^{(s)} z^m \dfrac{\partial}{\partial z_s}\in H^{n-1}(W,\theta)$: 
\begin{align*}
    L_{\xi_0}(X)&=\sum_{s=1}^n\sum_{\substack{m \in \mathbb Z^n\\ \forall i, m_i<0}}b_m^{(s)} L_{\xi_0}\left(z^m \dfrac{\partial}{\partial z_s} \right) = \sum_{j=1}^n\sum_{s=1}^n\sum_{\substack{m \in \mathbb Z^n\\ \forall i, m_i<0}}b_m^{(s)}\lambda_j [z_j\cfrac{\partial}{\partial z_j},z^m\dfrac{\partial}{\partial z_s}]\\
    &=\sum_{j=1}^n\sum_{s=1}^n\sum_{\substack{m \in \mathbb Z^n\\ \forall i, m_i<0}}b_m^{(s)}\lambda_j (m_j \dfrac{z^{m+e_j}}{z_j}\dfrac{\partial}{\partial z_s}-\delta_{s,j} \dfrac{z^{m+e_j}}{z_j} \dfrac{\partial}{\partial z_s})\\
    &=\sum_{s=1}^n\sum_{\substack{m \in \mathbb Z^n\\ \forall i, m_i<0}}b_m^{(s)}\left((m,\lambda)- \lambda_s \right )z^m\dfrac{\partial}{\partial z_s}=\sum_{s=1}^n\sum_{\substack{m \in \mathbb Z^n\\ \forall i, m_i<0}}b_m^{(s)} (m-e_s,\lambda) z^m\dfrac{\partial}{\partial z_s}.
\end{align*}
The exact preceding arguments for the sheaf of holomorphic functions still work in that case since if $m\in \mathbb Z^n$ is a tuple whose coefficients are all strictly negative, $m-e_s$ is also such a tuple. This proves that $L_{\xi_0} :H^{n-1}(W,\theta)\to H^{n-1}(W,\theta)$ is an isomorphism. \\
In the more general case where $\xi$ can be written as
$$\sum_{j=1}^n\lambda_jz_j \dfrac{\partial}{\partial z_j}+\sum_{(j,m)\in R} a_{j,m}z^m\dfrac{\partial}{\partial z_j},$$
we will proceed in exactly the same way as we did before for the sheaf $\sigma$. We only need to prove that $L_{\xi}:H^{1}(W,\theta)\to H^{1}(W,\theta)$ is injective when $n=2$.
We write for $X=\sum_{\substack{1\leq s \leq 2\\m \in \mathbb Z^2\\ \forall i, m_i<0}}b_m^{(s)} z^m \dfrac{\partial}{\partial z_s}\in H^{1}(W,\theta)$: 

\begin{align*}
        L_{\xi}(X)&=\sum_{\substack{1 \leq s \leq 2 \\m \in \mathbb Z^2\\ \forall i, m_i<0}}b_m^{(s)} (m-e_s,\lambda) z^m\dfrac{\partial}{\partial z_s}+\sum_{\substack{1 \leq s \leq 2 \\(2,l)\in R\\m \in \mathbb Z^2\\ \forall i, m_i<0}}b_{m}^{(s)}a_{2,l}[z^l\dfrac{\partial}{\partial z_2}, z^m\dfrac{\partial}{\partial z_s}]\\
        &=\sum_{\substack{1 \leq s \leq 2 \\m \in \mathbb Z^2\\ \forall i, m_i<0}}b_m^{(s)} (m-e_s,\lambda) z^m\dfrac{\partial}{\partial z_s}+\sum_{\substack{1 \leq s \leq 2 \\(2,l)\in R\\m \in \mathbb Z^2\\ \forall i, m_i<0}}b_{m}^{(s)}a_{2,l}m_2z^{m+l_1e_1-e_2}\dfrac{\partial}{\partial z_s}\sum_{\substack{(2,l)\in R\\m \in \mathbb Z^2\\ \forall i, m_i<0}}b_{m}^{(1)}a_{2,l}l_1z^{m+l_1e_1-e_1}\dfrac{\partial}{\partial z_2}.
    \end{align*}
    The same idea of proof applies : we prove that $b^{(s)}_{m}=0$ for $s=1$, $m_1<0$ and $m_2=-1$ then for $s=1$, $m_1<0$ and $m_2<0$ by induction on $m_2 \in \mathbb Z_{<0}$. Then we prove that $b^{(s)}_{m}=0$ for $s=2$, $m_1<0$ and $m_2=-1$ then for $s=2$, $m_1<0$ and $m_2<0$ by induction on $m_2 \in \mathbb Z_{<0}$. This concludes the proof of $(iii)$.\\
$(iv)$ : First assume that $\xi$ is equal to $\xi_0$. We compute, for $f=\sum_{m\in \mathbb N^n}a_m z^m \in H^{0}(W,\sigma)$: 
\begin{align*}
    L_\xi(f)=\sum_{m\in \mathbb N^n\setminus 0}a_m (m,\lambda) z^m.
\end{align*}
If $L_\xi(f)=0$, then since $(m,\lambda)$ is not zero for every $m\in \mathbb N^n\setminus0$ (again, $0 \notin \mathcal{H}(\lambda)$), it comes that $f$ is constant equal to $a_0 \in \mathbb C$. The converse is also true. This proves that $\ker(L_{\xi_0}:H^0(W, \sigma) \to H^0(W, \sigma))$ is equal to $\text{Vect}_\mathbb C(1)$.\\
Let $g= \sum_{m\in \mathbb N^n\setminus0}b_m z^m \in H^0(W,\sigma)$. As before, if we let $a_m=\dfrac{b_m}{(m, \lambda)}$ for $m\in \mathbb N^n \setminus 0$, the power serie $\sum_{m\in \mathbb N^n\setminus0}a_m z^m$ defines an element $f$ of $H^0(W, \sigma)$ which satisfies $L_\xi(f)=g$. This proves that $\text{Vect}_\mathbb C(1)$ is a complementary subspace to $\text{Im}(L_{\xi_0}:H^0(W, \sigma) \to H^0(W, \sigma))$ in $H^0(W, \sigma)$.\\
From the long exact sequence associated to \eqref{eq:1}, it comes on one hand 
$$\dim(H^0(S^{2n-1}, \sigma^{tr}_{\mathcal{F}_0(\xi_0)}))=\dim(H^0(W, \sigma^{tr}_{\mathcal{F}(\xi_0)}))=\dim(\ker(L_{\xi_0}:H^0(W, \sigma) \to H^0(W, \sigma)))=1$$
and on the other hand, since $\delta : H^0(W, \sigma) \to H^1(W, \sigma^{tr}_{\mathcal{F}(\xi_0)})$ is surjective because $L_\xi:H^1(W, \sigma) \to H^1(W, \sigma)$ is injective :
$$\dim(H^1(S^{2n-1}, \sigma^{tr}_{\mathcal{F}_0(\xi_0)}))=\dim(H^1(W, \sigma^{tr}_{\mathcal{F}(\xi_0)}))=\dim(\text{Im}(\delta))=1$$
because $\dim(\text{Im}(\delta))$ is equal to the dimension of a complementary subspace to $\ker(\delta)=\text{Im}(L_{\xi_0}:H^0(W, \sigma) \to H^0(W, \sigma))$ in $H^0(W, \sigma)$, which is one by the above. Therefore, since $$H^k(S^{2n-1}, \sigma^{tr}_{\mathcal{F}_0(\xi_0)})=H^k(W, \sigma^{tr}_{\mathcal{F}(\xi_0)})=0$$
for every $k \geq 2$ by point $(iii)$, it comes $\sum_{k\in \mathbb N}(-1)^k\dim (H^k(S^{2n-1}, \sigma^{tr}_{\mathcal{F}_0(\xi_0)}))=0$. This number is the index of an elliptic complex so it is constant under deformation (see \cite{wells_differential_2008}, \cite{atiyah_index_1963}, \cite{atiyah_index_1968} for example).
Therefore, even in the general case when $\xi=\xi_0+\xi_R$, since $$H^k(S^{2n-1}, \sigma^{tr}_{\mathcal{F}_0(\xi)})=H^k(W, \sigma^{tr}_{\mathcal{F}(\xi)})=0$$
for every $k \geq 2$ by point $(iii)$, it comes
$$\dim (H^0(W, \sigma^{tr}_{\mathcal{F}(\xi)}))=\dim (H^0(S^{2n-1}, \sigma^{tr}_{\mathcal{F}_0(\xi)}))=\dim (H^1(S^{2n-1}, \sigma^{tr}_{\mathcal{F}_0(\xi)}))=\dim (H^1(W, \sigma^{tr}_{\mathcal{F}(\xi)})).$$
By the same upper-semi continuity of $\dim (H^0(W, \sigma^{tr}_{\mathcal{F}(\xi)}))$ argument, this number (if the components $a_{j,m}$ of $\xi_R$ are initially chosen sufficiently small) is less or equal to $1$. But it is also greater or equal, thus equal, to $1$ because $$\text{Vect}_\mathbb C(1)\subset\ker(L_{\xi}:H^0(W, \sigma) \to H^0(W, \sigma))= H^0(W, \sigma^{tr}_{\mathcal{F}(\xi)}).$$
Thus, since $\delta$ is surjective (for the same reasons than before by point $(iii)$), the cokernel of $\text{Im}(L_\xi:H^0(W, \sigma) \to H^0(W,\sigma))$ is of dimension $1$. But we know that $1 \notin \text{Im}(L_\xi:H^0(W, \sigma) \to H^0(W,\sigma))$ because
$$L_\xi\left(\sum_{m\in \mathbb N^n}b_m z^m \right)=\sum_{m \in \mathbb N^n\setminus 0}b_m(m,\lambda)z^m+\sum_{\substack{m\in \mathbb N^n\setminus 0\\(j,l)\in R}}b_mm_ja_{j,l}z^{m+l-e_j}.$$
Therefore $\text{coker}(L_\xi:H^0(W, \sigma) \to H^0(W,\sigma)) = \text{Vect}_\mathbb C(1)$ i.e. 
$$H^0(W,\sigma)=\text{Im}(L_\xi)\oplus \text{Vect}_\mathbb C(1) $$
and $L_\xi:H^0(W, \sigma) \to H^0(W,\sigma)$ surjects on the space of holomorphic functions on $\mathbb C^n$ vanishing at $0$. This proves $(iv)$.\\
$(v)$ : We first prove that $\restr{L_\xi}{\mathcal{g}_{\lambda}^\perp}$ is injective. We write for $X=\sum_{\substack{1 \leq s \leq n \\m \in \mathbb N^n\\ \lambda_s \neq (m, \lambda)}}b_m^{(s)} z^m \dfrac{\partial}{\partial z_s}\in \mathcal{g}_\lambda^\perp$: 

\begin{align*}
        L_{\xi}(X)=\sum_{\substack{1 \leq s \leq n \\m \in \mathbb N^n\\ \lambda_s\neq (m,\lambda)}}b_m^{(s)} ((m,\lambda)-\lambda_s) z^m\dfrac{\partial}{\partial z_s}&+\sum_{\substack{1 \leq s,j \leq n \\m \in \mathbb N^n\\(j,l)\in R\\\lambda_s \neq(m, \lambda)}}b_{m}^{(s)}a_{j,l}m_jz^{m+l-e_j}\dfrac{\partial}{\partial z_s}\\&-\sum_{\substack{1 \leq s,j \leq n \\m \in \mathbb N^n\\(j,l)\in R\\\lambda_s \neq(m, \lambda)}}b_{m}^{(s)}a_{j,l}l_sz^{m+l-e_s}\dfrac{\partial}{\partial z_j}.
    \end{align*}
If $L_\xi(X)=0$, on can prove, for $s \in \llbracket1,n\rrbracket$ fixed, that $b^{(s)}_m=0$ for every $m\in \mathbb N^n$ such that $\lambda_s\neq (m, \lambda)$. The idea is the following.
Let $s=1$. The monomial $z^{e_n}\dfrac{\partial}{\partial z_1}$ appears only in the first sum by definition of $R$, with coefficient $b^{(1)}_{e_n}((e_n,\lambda)-\lambda_1)$. By properties of power series, $b^{(1)}_{e_n}=0$. By decreasing induction on $k\in \llbracket1,n \rrbracket$, we prove that $b^{(1)}_{e_k}=0$ for every $k\in \llbracket1,n \rrbracket$. 
The monomial $z^{e_k}\dfrac{\partial}{\partial z_1}$ appears in the first sum, with coefficient $b^{(1)}_{e_k}((e_k,\lambda)-\lambda_1)$, and in the second sum it corresponds only to $j=k+1$ and thus $m=e_k$ so it does not appear by induction hypothesis. By properties of power series, $b^{(1)}_{e_k}=0$.
Therefore, $b^{(1)}_m=0$ if $|m|=1$.
The monomial $z^{e_n+e_n}\dfrac{\partial}{\partial z_1}$ appears only in the first sum by definition of $R$, because in the second it would correspond to $|l|=1$ (since $b^{(1)}_m=0$ if $|m|=1$) thus to $l=e_{j-1}$ by definition of $R$. But this is impossible as $m+e_{j-1}=e_j+2e_n$. Therefore, $b^{(1)}_{e_n +e_n}=0$.
We use this result to prove that $b^{(1)}_{e_{k} +e_n}=0$ by decreasing induction on $k\in \llbracket1,n\rrbracket$. Then we prove that $b^{(1)}_{e_{k} +e_{n-1}}=0$ for every $k\in \llbracket1,n-1\rrbracket$ and eventually $b^{(1)}_{m}=0$ if $|m|=0$.
Said differently, one can prove that $b^{(1)}_{m}=0$ in the "reversed" lexicographic order (ie if one reads from right to left). Then $b^{(s)}_{m}=0$ for every $s \in \llbracket
1,n \rrbracket$ and $m\in \mathbb N^n$.\\
We now prove that $\restr{L_\xi}{\mathcal{g}_\lambda^\perp} : \mathcal{g}_\lambda^\perp \to \mathcal{g}_\lambda^\perp$ is surjective. First assume $\xi$ is equal to the diagonal vector field $\xi_0$.
If $Y=\sum_{\substack{1 \leq s \leq n\\m \in \mathbb N^n\\ \lambda_s\neq (m,\lambda)}}c_m^{(s)} z^m \dfrac{\partial}{\partial z_s}\in \mathcal{g}_\lambda^\perp$, then $\sum_{\substack{1 \leq s \leq n\\m \in \mathbb N^n\\ \lambda_s\neq (m,\lambda)}}\dfrac{c_m^{(s)}}{(m,\lambda)-\lambda_s} z^m \dfrac{\partial}{\partial z_s}$ defines an element of $\mathcal{g}_\lambda^\perp\subset H^0(W,\theta)$ since for every $s \in \llbracket1,n \rrbracket$, there exists $M_s \in \mathbb N$ such that 
$$|m|\geq M_s \implies |(m,\lambda)-\lambda_s|\geq 1$$
because there is only finitely many resonances. Such element $X$ satisfies 
\begin{align*}
    L_{\xi_0}(X)=\sum_{\substack{1 \leq s \leq n\\m \in \mathbb N^n\\ \lambda_s\neq (m,\lambda)}}\dfrac{c_m^{(s)}}{(m,\lambda)-\lambda_s} ((m,\lambda)-\lambda_s) z^m\dfrac{\partial}{\partial z_s}=Y.
\end{align*}
In the more general case where $\xi=\xi_0+\xi_R$, we prove that $\mathcal{g}_\lambda^\perp/L_\xi(\mathcal{g}_\lambda^\perp)=0$ (this makes sense by point $(iv)$ of Proposition \ref{prop:4.5}).
In order to do so, remark that the map $\delta':H^0(W, \theta) \to H^1(W, \theta^\xi)$ is surjective by the long exact sequence associated to \eqref{eq:2} and by the previous point $(iii)$. Therefore :
\begin{align*}
    H^1(W, \theta^\xi)&=\text{Im}(\delta')\cong H^0(W,\theta)/\ker(\delta')=H^0(W,\theta)/\text{Im}(L_\xi)\\
    &=(\mathcal{g}_\lambda\oplus\mathcal{g}_\lambda^\perp)/(L_\xi(\mathcal{g}_\lambda)\oplus L_\xi(\mathcal{g}_\lambda^\perp))\cong(\mathcal{g}_\lambda/L_\xi(\mathcal{g}_\lambda))\times (\mathcal{g}_\lambda^\perp/L_\xi(\mathcal{g}_\lambda^\perp))
\end{align*}
because of the following fact :
\begin{claim}
    Let $E$ and $F$ vector subspaces of a vector space $G$. Let $E'$ (respectively $F'$) a vector subspace of $E$ (respectively $F$). Assume $E\cap F=\{0\}$.\\
    Then the well-defined natural map
    $$\left \{\begin{array}{ccl}
        E/E' \times F/F' & \to   & (E+E')/(F+F')\\
        (\overline{x}, \hat{y}) & \mapsto  & \widetilde{x+y}
    \end{array} \right .$$
    is an isomorphism.
\end{claim}
\noindent Also, because $\mathcal{g}_\lambda$ is a finite dimensional vector space and $\restr{L_\xi}{\mathcal{g}_\lambda^\perp} : \mathcal{g}_\lambda^\perp \to \mathcal{g}_\lambda^\perp$ is injective, it comes $$\mathcal{g}_\lambda/L_\xi(\mathcal{g}_\lambda) \cong\ker(\restr{L_\xi}{\mathcal{g}_\lambda}) = \ker(L_\xi : H^0(W;\theta) \to H^0(W, \theta))=H^0(W, \theta^\xi)$$
and thus
\begin{align*}
    H^1(W, \theta^\xi)\cong H^0(W,\theta^\xi)\times (\mathcal{g}_\lambda^\perp/L_\xi(\mathcal{g}_\lambda^\perp)).
\end{align*}
Now, recall that $H^k(S^{2n-1}, \sigma^{tr}_{\mathcal{F}_0})=H^k(W, \sigma^{tr}_{\mathcal{F}})=0$ for every $k\geq 2$ by point $(iii)$, and consider the beginning of the long exact sequence of cohomology groups associated to \eqref{eq:3}:
\[
\begin{tikzcd}
0 \ar[r] & H^0(W,\sigma^{tr}_\mathcal{F}) \ar[r,"\alpha_0"] & H^0(W,\theta^\xi) \ar[r, "\beta_0"] & H^0(W, \theta^{tr}_\xi) \ar[dll] \\
& H^1(W,\sigma^{tr}_\mathcal{F}) \ar[r,"\alpha_1"] & H^1(W,\theta^\xi) \ar[r,"\beta_1"] & H^1(W, \theta^{tr}_\xi) \ar[r] & 0.
\end{tikzcd}
\]
It comes, for $i\in \{0,1\}$ that 
$$H^i(W, \theta^\xi)/\text{Im}(\alpha_i)=H^i(W, \theta^\xi)/\ker(\beta_i)\cong \text{Im}(\beta_i).$$ 
Also we know that for every $k \in \mathbb N$, $$H^k(W, \sigma^{tr}_\mathcal{F})\cong H^k(S^{2n-1}, \sigma^{tr}_{\mathcal{F}_0}) \quad \text{ and } \quad H^k(W, \theta^{tr}_\mathcal{F})\cong H^k(S^{2n-1}, \theta^{tr}_{\mathcal{F}_0})$$ 
are finite dimensional vector spaces (see \cite{duchamp_deformation_1979} for example). Since $\alpha_i$ is defined on $H^i(W, \sigma^{tr}_\mathcal{F})$ and $\beta_i$ takes values in $H^i(W, \theta^{tr}_\mathcal{F})$, their images are finite dimensional vector spaces. Thus $H^0(W, \theta^\xi)$ and $H^1(W, \theta^\xi)$ are finite dimensional vector spaces, whose dimension satisfy 
$$\sum_{i=0}^1(-1)^{i}\dim(H^i(W, \sigma^{tr}_\mathcal{F}))+\sum_{i=0}^1(-1)^{i}\dim(H^i(W, \theta^\xi))+\sum_{i=0}^1(-1)^{i}\dim(H^i(W, \theta^{tr}_\mathcal{F}))=0$$
by the above exact sequence and the rank-nullity theorem.
We have already discussed the fact that the first term remains constant equal to $0$ after deformation since each $H^i(W, \sigma^{tr}_\mathcal{F})$ is isomorphic to $H^i(S^{2n-1}, \sigma^{tr}_{\mathcal{F}_0})$ and the sheaf $\sigma^{tr}_{\mathcal{F}_0}$ admits a resolution by an elliptic complex. 
This is also true if we replace the letter $\sigma$ by $\theta$ (see \cite{duchamp_deformation_1979}). Therefore, the third term also remains constant after deformation. But we know that, in the case where $\xi=\xi_0$, the map $\restr{L_\xi}{\mathcal{g}_\lambda^\perp}:\mathcal{g}_\lambda^\perp \to \mathcal{g}_\lambda^\perp$ is an isomorphism. 
Therefore since 
$$H^1(W, \theta^\xi)\cong H^0(W,\theta^\xi)\times (\mathcal{g}_\lambda^\perp/L_\xi(\mathcal{g}_\lambda^\perp))\cong H^0(W,\theta^\xi),$$
it comes by the previous equality that the third term remains constant equal to $0$ after deformation. As a result, in the general case where $\xi=\xi_0+\xi_R$, by the previous equality:
$$\dim(H^0(W, \theta^\xi))=\dim(H^1(W, \theta^\xi)).$$
By the above, the latter implies that $\mathcal{g}_\lambda^\perp/L_\xi(\mathcal{g}_\lambda^\perp)$ is a finite dimensional vector space of dimension $0$, which concludes.
\end{proof}
\begin{remark}
    This proof gives also :
    \begin{itemize}
        \item For every $k \geq2$, $$H^k(S^{2n-1}, \sigma^{tr}_{\mathcal{F}_0})=0 \quad \text{ and } \quad H^k(S^{2n-1}, \theta^{tr}_{\mathcal{F}_0})=0;$$
        \item $\dim(H^0(S^{2n-1}, \sigma^{tr}_{\mathcal{F}_0}))=\dim(H^1(S^{2n-1}, \sigma^{tr}_{\mathcal{F}_0}))=1$;
        \item $\dim(H^0(S^{2n-1}, \theta^{tr}_{\mathcal{F}_0}))=\dim(H^1(S^{2n-1}, \theta^{tr}_{\mathcal{F}_0}))$.
    \end{itemize}
\end{remark}
\begin{corollary}
    \leavevmode
    \begin{enumerate}[label=(\roman*)]
        \item The maps $$\delta:H^0(W, \sigma) \to H^1(W, \sigma^{tr}_\mathcal{F}), \; \delta' :H^0(W, \theta) \to H^1(W, \theta^{\xi}), \; p :H^1(W, \theta^\xi) \to H^1(W, \theta^{tr}_\mathcal{F})$$ are surjective ;
        \item The map
        $$\restr{\delta'}{T_0S \oplus \text{Vect}_\mathbb C(\xi)}:T_0S \oplus \text{Vect}_\mathbb C(\xi) \to H^1(W, \theta^\xi)$$
        is an isomorphism.
    \end{enumerate}
\end{corollary}
\begin{proof}
    $(i)$ : We have already discussed the surjectivity of $\delta$ and $\delta'$ during the previous proof. The fact that $p$ is surjective comes from the exact sequence and the previous remark.\\
    $(ii)$ : Let $g \in H^{1}(W, \theta^\xi)$. Since $\delta' :H^0(W, \theta) \to H^1(W, \theta^{\xi})$ is surjective, and because of point $(v)$ the previous proposition and Proposition \ref{prop:4.5}, there exists $X^r\in \mathcal{g}_\lambda $ and $X^{nr} \in \mathcal{g}_\lambda^\perp=L_\xi(\mathcal{g}_\lambda^\perp)$ such that $$g=\delta'(X^r)+\delta'(X^{nr})=\delta'(X^r)$$
    because $\ker(\delta')=\text{Im}(L_\xi : H^0(W,\theta) \to H^0(W,\theta))$ by the long exact sequence. By definition of $S$, $X^r$ writes as $V+a\xi+L_\xi(Y^r)$ where $V\in T_0S, \,a \in \mathbb C, \; Y^r \in \mathcal{g}_\lambda$ so eventually $g=\delta'(V+a\xi)$. This proves the surjectivity.\\
    Assume now that $\delta'(V+a\xi)=0$, where $V\in T_0S, \,a \in \mathbb C$, i.e. there exists $X=X^r +X^{nr} \in \mathcal{g}_\lambda \oplus \mathcal{g}_\lambda^\perp$ such that $V+a\xi=L_\xi(X^r)+L_\xi(X^{nr})$. Necessarily, $L_\xi(X^{nr})=0$. Therefore, by Corollary \ref{cor:4.5.1}, it comes $V=0$ and $a=0$.\\
    We could have also proven that $T_0S\oplus \text{Vect}_\mathbb C(\xi)$ is a complementary subspace to $\ker(\delta')=\text{Im}(L_\xi:H^0(W,\theta)\to H^0(W, \theta))$ since $\delta'$ is surjective. It is immediate by the previous proposition :
    \begin{align*}
        H^0(W,\theta)&=\mathcal{g}_\lambda\oplus\mathcal{g}_\lambda^\perp=\left (T_0S\oplus\text{Vect}_\mathbb C(\xi)\oplus L_\xi(\mathcal{g}_\lambda)\right)\oplus L_\xi(\mathcal{g}_\lambda^\perp)\\
        &=\left (T_0S\oplus\text{Vect}_\mathbb C(\xi)\right)\oplus \left(L_\xi(\mathcal{g}_\lambda)\oplus L_\xi(\mathcal{g}_\lambda^\perp)\right)=\left (T_0S\oplus\text{Vect}_\mathbb C(\xi)\right)\oplus \text{Im}(L_\xi).
    \end{align*}
\end{proof}
We eventually prove that $\restr{p\circ \delta'}{T_0S} : T_0S \to H^1(W, \theta^{tr}_\mathcal{F})$ is an isomorphism. Let $g \in H^{1}(W, \theta^\xi)$. As $p\circ\delta':H^0(W, \theta) \to H^1(W, \theta^{tr}_\mathcal{F})$ is surjective by composition, there exist $V \in T_0S, \; a \in \mathbb C, \; X\in H^0(W, \theta)$ such that
$$g=(p\circ \delta') (V+a\xi+L_\xi(X))=(p\circ \delta') (V)+p(\delta'(a\xi))$$
because $\ker(\delta')=\text{Im}(L_\xi : H^0(W,\theta) \to H^0(W,\theta))$ by the long exact sequence.
Also, by the surjectivity of $\delta : H^0(W, \sigma) \to H^0(W,\sigma)$ and commutativity of the diagram of long exact sequence of cohomology groups, the kernel of $p$ is equal to
\begin{align*}
   \ker(p) &=\text{Im} \left(m_\xi: H^1(W, \sigma^{tr}_\mathcal{F}) \to H^1(W, \theta^\xi)\right)=\text{Im} \left(m_\xi\circ \delta: H^0(W, \sigma) \to H^1(W, \theta^\xi)\right)\\
    &=\text{Im} \left(\delta'\circ m_\xi: H^0(W, \sigma) \to H^1(W, \theta^\xi)\right)
\end{align*}
which therefore gives $g=(p\circ \delta') (V)$ and thus the surjectivity of the map $\restr{p\circ \delta'}{T_0S} : T_0S \to H^1(W, \theta^{tr}_\mathcal{F})$.\\
Now assume $V\in T_0S$ satisfies $p(\delta'(V))=0$, i.e. $\delta'(V)$ is equal to some $\delta'(m_\xi(f))$ where $f \in H^0(W, \sigma)$ by the above. Write $f=a +h$ where $a\in \mathbb C$ is the constant term of $f$ and $h\in H^0(W, \sigma)$ is a holomorphic function on $\mathbb C^n$ vanishing at $0$.
Then \begin{align*}
    \delta'(V)&=\delta'(a\xi +h \xi)=\delta'(a\xi)+\delta'(m_\xi(h))\\
    &=\delta'(a\xi)+m_\xi(\delta(h))
\end{align*}
but we also know that $\ker(\delta)=\text{Im}(L_\xi : H^0(W,\sigma) \to H^0(W,\sigma))$ coincides with the set of holomorphic functions on $\mathbb C^n$ vanishing at $0$, by point $(iv)$ of the previous proposition.
Therefore $\delta'(V)=\delta'(a\xi)$ which implies that $V-a\xi\in \ker(\delta')=\text{Im}(L_\xi : H^0(W,\theta) \to H^0(W,\theta))$ and thus, by Corollary \ref{cor:4.5.1} and definition of $S$, $V=0$ and $a=0$.
This proves injectivity of the map $\restr{p\circ \delta'}{T_0S} : T_0S \to H^1(W, \theta^{tr}_\mathcal{F})$ and the result.
\bibliographystyle{abbrv}
\bibliography{article_2.bib}
\end{document}